\documentclass[reqno,11pt]{amsart}
\usepackage{a4wide,color,eucal,enumerate,mathrsfs}
\usepackage[normalem]{ulem}
\usepackage{amsmath,amssymb,epsfig,amsthm} 
\usepackage[latin1]{inputenc}
\usepackage{graphicx}
\usepackage{psfrag}
\numberwithin{equation}{section}

\newtheorem{theorem}{Theorem}[section]

\newtheorem{corollary}[theorem]{Corollary}
\newtheorem{lemma}[theorem]{Lemma}
\newtheorem{proposition}[theorem]{Proposition}

\theoremstyle{definition}
\newtheorem{definition}[theorem]{Definition}

\newtheorem{example}[theorem]{Example}

\theoremstyle{remark}

\DeclareMathOperator{\diam}{diam}
\DeclareMathOperator{\dist}{dist}

\newcommand{\N}{\mathbb{N}}

\newcommand{\R}{\mathbb{R}}

\def\dist{{\mathop\mathrm{\,dist\,}}}

\def\bint{{\ifinner\rlap{\bf\kern.35em--}
\int\else\rlap{\bf\kern.45em--}\int\fi}\ignorespaces}

\def\bbint{{\ifinner\rlap{\bf\kern.35em--}
\hspace{0.078cm}\int\else\rlap{\bf\kern.45em--}\int\fi}\ignorespaces}

\def\diam{{\mathop\mathrm{\,diam\,}}}

\def\bint{{\ifinner\rlap{\bf\kern.35em--}
\int\else\rlap{\bf\kern.45em--}\int\fi}\ignorespaces}

\begin{document}

\title[Strong $BV$-extension and $W^{1,1}$-extension domains]{Strong $BV$-extension and $W^{1,1}$-extension domains}

\author{Miguel Garc\'ia-Bravo}
\author{Tapio Rajala}

\address{University of Jyvaskyla \\
         Department of Mathematics and Statistics \\
         P.O. Box 35 (MaD) \\
         FI-40014 University of Jyvaskyla \\
         Finland}
         
\email{miguel.m.garcia-bravo@jyu.fi}         
\email{tapio.m.rajala@jyu.fi}

\thanks{The authors acknowledge the support from the Academy of Finland, grant no. 314789.}
\date{\today}


\begin{abstract}
We show that a bounded domain in a Euclidean space is a $W^{1,1}$-extension domain if and only if it is a strong $BV$-extension domain. In the planar case, bounded and strong $BV$-extension domains are shown to be exactly those $BV$-extension domains for which the set $\partial\Omega \setminus \bigcup_{i} \overline{\Omega}_i$ is purely $1$-unrectifiable, where $\Omega_i$ are the open connected components of $\R^2\setminus\overline{\Omega}$.
\end{abstract}


\maketitle

\tableofcontents

\section{Introduction}

Let $\Omega\subset\R^n$ be a domain for some $n\geq 2$. For every $1\leq p\leq \infty$, we define the Sobolev space $W^{1,p}(\Omega)$ to be
$$W^{1,p}(\Omega)=\{u\in L^p(\Omega):\, \nabla u\in L^p(\Omega;\R^n)\}, $$
where $\nabla u$ denotes the distributional gradient of $u$. We equip this space with the non-homogeneous norm
$$\|u\|_{W^{1,p}(\Omega)}= \|u\|_{L^p(\Omega)}+\|\nabla u\|_{L^p(\Omega)}. $$
We say that $\Omega$ is a $W^{1,p}$-extension domain if there exists an  operator $T\colon W^{1,p}(\Omega)\to W^{1,p}(\R^n)$ and a constant $C>0$ so that 
$$ \|Tu\|_{W^{1,p}(R^n)}\leq C \|u\|_{W^{1,p}(\Omega)}$$
and $Tu|_\Omega = u$ for every $u\in W^{1,p}(\Omega)$. 
We denote the minimal constant $C$ above by $\|T\|$.
We point out that by the results from \cite{HKT2008,S2006}, for $p>1$ one can always assume the operator $T$ to be linear, and also for the case of bounded simply connected planar domains if $p=1$ by \cite{KRZ}. It is not yet known if this is the case for general domains when $p=1$.

It is well-known from the works of Calder\'on and Stein \cite{cal1961,stein} that Lipschitz domains are $W^{1,p}$-extension domains for every $p\geq 1$. Moreover, Jones showed in \cite{jo1981} that
every uniform domain $\Omega\subset \R^n$ is a $W^{1,p}$-extension domain for all $p\geq 1$. However, these conditions are not necessary for a domain to be a Sobolev extension domain. For bounded simply connected planar domains a geometric characterization of Sobolev extension domains by means of a curve condition has been given in the works \cite{sh2010,KRZ2015,KRZ}. Namely, for the $W^{1,1}$ case we have the following: A bounded planar simply connected domain $\Omega$ is a $W^{1,1}$-extension domain if and only if for every $x,y\in \Omega^c$ there exists a curve  $ \gamma \subset \Omega^c$ connecting $x$ and $y$ with 
\begin{equation}\label{eq:curvep1}
\ell(\gamma) \le C|x-y|, \text{ and } \mathcal H^1(\gamma \cap \partial \Omega) = 0.
\end{equation}

A typical example of a simply connected planar domain $\Omega$ which is not a $W^{1,p}$-extension domain for any $p\geq 1$ is the slit disk $D=\{(x,y)\in\R^2:\, x^2+y^2<1\}\setminus ([0,1)\times \{0\})\}$. However, by the results of \cite{KMS2010}, knowing that the complement is quasiconvex is enough to ensure that $D$ is a $BV$-extension domain.

Recall that
$$BV(\Omega)=\{u\in L^1(\Omega):\, \| D u\|(\Omega)<\infty\}$$
is the space of functions of bounded variation where
$$\Vert D u\Vert (\Omega)=\sup\left\{\int_{\Omega}u \,\text{div} (v)\, dx:\, v\in C^{\infty}_{0}(\Omega;\R^n) ,\, |v|\leq 1\right\}$$
denotes the total variation of $u$ on $\Omega$.
We endow this space with the norm $ \Vert u\Vert_{BV(\Omega)}=\Vert u\Vert_{L^1(\Omega)}+\Vert D u\Vert (\Omega).$ Note that $\|Du\|$ is a Radon measure on $\Omega$ that is defined for every set $F\subset\Omega$ as \[
\|Du\|(F)=\inf\{\|Du\|(U):\, F\subset U\subset\Omega,\;U\;\text{open}\}.
\]

We say that $\Omega$ is a $BV$-extension domain if there exists a constant $C>0$ and a (not necessarily linear) extension operator $T\colon BV(\Omega)\to BV(\R^n)$ so that $Tu|_{\Omega}=u$ and 
$$\|Tu\|_{BV(\R^n)}\leq C\|u\|_{BV(\Omega)} $$
for all $u\in BV(\Omega)$ and where $C>0$ is an absolute constant, independent of $u$.
Let us point out that $\Omega$ being a $W^{1,1}$-extension domain always implies that it is also a $BV$-extension domain (see \cite[Lemma 2.4]{KMS2010}).

Our first main result is the characterization of bounded $W^{1,1}$-extension domains in terms of strong extendability of $BV$-functions, or equivalently, in terms of strong extendability of sets of finite perimeter.
The equivalence between strong extendability of $BV$-functions and strong extendability of sets of finite perimeter is inspired by the work of Mazy'a and Burago \cite{BM1967} (see also \cite[Section 9.3]{mazya}). They showed that for all $u\in L^{1}_{\text{loc}}(\Omega)$ with finite total variation we may find an extension $Tu\in L^{1}_{\text{loc}}(\R^n)$ with $\|D(Tu) \|(\R^n) \leq C \|Du\|(\Omega)$, for some constant $C>0 $, if and only if any set $E\subset\Omega$ of finite perimeter in $\Omega$ admits an extension $\widetilde E\subset\R^n$ satisfying $\widetilde E\cap\Omega=E$ and $P(\widetilde E,\R^n)\leq C P(E,\Omega)$ where $C>0$ is some constant. Recall that a Lebesgue measurable subset $E\subset \R^n$ has finite perimeter in $\Omega$ if $\chi_E\in BV(\Omega)$, where $\chi_E$ denotes the characteristic function of the set $E$. We set $P(E,\Omega)=\|D \chi_E\|(\Omega)$ and call it the perimeter of $E$ in $\Omega$. If a set $E$ does not have finite perimeter in $\Omega$ we set $P(E,\Omega)=\infty$.

Before stating our characterization, we introduce the terminology of strong extendability, following  \cite{HKLL2016} and  \cite{L2015}.

\begin{definition}[Strong $BV$-extension domain]\label{def:strongBV}
A domain $\Omega \subset \mathbb R^n$ is called a \emph{strong $BV$-extension domain} if there exists a constant $C>0$ so that for any
$u \in BV(\Omega)$ there exists $Tu \in BV(\mathbb R^n)$ with $Tu|_\Omega = u$,
$\|Tu\|_{BV(\mathbb R^n)} \le C\|u\|_{BV(\Omega)}$, and $\|D(Tu)\|(\partial\Omega) = 0$.
\end{definition}

In the spirit of Definition \ref{def:strongBV}, we define the analogous concept for sets of finite perimeter.

\begin{definition}[Strong extension property for sets of finite perimeter]\label{def:strongperi}
A domain $\Omega \subset \mathbb R^n$ is said to have the \emph{strong extension property for sets of finite perimeter} if
there exists a constant $C>0$ so that for any set $E\subset \Omega$ of finite perimeter in $\Omega$ there exists a set  $\widetilde E\subset \R^n$ such that
\begin{enumerate}
    \item[(PE1)] $\widetilde E \cap \Omega = E$  modulo measure zero sets,
    \item[(PE2)] $P(\widetilde E,\R^n)\leq C P(E,\Omega) $, and
    \item[(PE3)] $\mathcal{H}^{n-1}(\partial^M \widetilde E \cap \partial \Omega)=0$.
\end{enumerate}
\end{definition}

With the above definitions we can state our first main result.

\begin{theorem}\label{thm:mainRn}
Let $\Omega\subset\R^n$ be a bounded domain. Then the following are equivalent:
\begin{enumerate}
    \item $\Omega$ is a $W^{1,1}$-extension domain.
    \item $\Omega$ is a strong $BV$-extension domain.
    \item $\Omega$ has the strong extension property for sets of finite perimeter.
\end{enumerate}
\end{theorem}

Our main motivation behind this theorem is to understand better the geometry of $W^{1,1}$-extension domains. From Theorem \ref{thm:mainRn} we see that for a bounded $W^{1,1}$-extension domain, except for a purely $(n-1)$-unrectifiable set, the boundary consists of points where the domain has density at most $1/2$. See Section \ref{sec:corollaries} for the proof of this.
In the same section we give an example showing that the above density bound is not sufficient to imply that a bounded $BV$-extension domain is a $W^{1,1}$-extension domain, even in the plane. 
Another corollary of Theorem \ref{thm:mainRn} is that 
for a bounded $W^{1,1}$-extension domain, again up to a 
purely $(n-1)$-unrectifiable set,
the boundary consists of points that are boundary points also for some component of the interior of the complement of the domain. 
In Section \ref{sec:corollaries} we provide also an example showing that in $\mathbb R^3$ this property does not characterize $W^{1,1}$-extension domains among bounded $BV$-extension domains.
However, our second main result states that in the planar case this is true.

\begin{theorem}\label{thm:planar}
Let $\Omega \subset \mathbb R^2$ be a bounded $BV$-extension domain. Then $\Omega$ is a $W^{1,1}$-extension domain if and only if the set
\[
 \partial \Omega \setminus \bigcup_{i \in I} \overline{\Omega_i}
\]
is purely $1$-unrectifiable, where $\{\Omega_i\}_{i\in I}$ are the connected components of $\mathbb R^2 \setminus \overline{\Omega}$.
\end{theorem}

Let us mention that Theorem \ref{thm:planar} recovers partly the theorems in \cite{KRZ}. Namely, it immediately follows that Jordan $BV$-extension domains are $W^{1,1}$-extension domains since the set required in Theorem \ref{thm:planar} to be purely unrectifiable, is indeed empty.
The curve characterization \eqref{eq:curvep1} also follows quite easily from Theorem \ref{thm:planar} using a small observation recorded in \cite{KRZ}.
Let us briefly sketch this. Since a $W^{1,1}$-extension domain is known to be a $BV$-extension domain, its complement is quasiconvex. Then, a quasiconvex curve
between two points in the complement can be modified to intersect the boundaries of each $\Omega_i$ at most twice (see Lemma \ref{lem:quasi.int.Omega_i}). Theorem \ref{thm:planar} now says that the rest of the curve intersects $\partial \Omega$ in a $\mathcal H^1$-measure zero set, giving condition \eqref{eq:curvep1}.
Conversely, \eqref{eq:curvep1} implies quasiconvexity, and hence that $\Omega$ is a $BV$-extension domain. For a simply connected $\Omega$, we can connect every pair of components $\Omega_i$ and $\Omega_j$ with a curve satisfying \eqref{eq:curvep1}. Since the set
\[
 \partial \Omega \setminus \bigcup_{i \in I} \overline{\Omega_i}.
\]
is contained in countably many of such curves by \cite[Lemma 4.6]{KRZ},
we see that it is purely $1$-unrectifiable.

Let us point out, however, that the extension operator that we construct in Theorem \ref{thm:planar}, is not always linear. One of the main points of \cite{KRZ} was to construct a linear extension operator. At the moment we
do not see how our construction could be modified to give a linear extension operator. Still, the general smoothing operator we use for proving Theorem \ref{thm:mainRn} (and Theorem \ref{thm:planar}) immediately gives the following.

\begin{corollary}
Suppose $\Omega \subset \mathbb R^n$ is a bounded strong $BV$-extension domains where the extension operator is linear. Then there exists a linear $W^{1,1}$-extension operator from $W^{1,1}(\Omega)$ to $W^{1,1}(\mathbb R^n)$.
\end{corollary}

Although not strictly used in our proofs, we include the following result for future use: Every $BV$-extension domain $\Omega\subset\R^n$ satisfies the measure density condition, that is, there exists a constant $c>0$ so that for every $x\in\overline{\Omega}$ and $r\in (0,1]$ we have
$|B(x,r)\cap \Omega|\geq c r^n .$ One may find this result in Section 2.2. The same conclusion for $W^{1,p}$
-extension domains with $1\leq p<\infty$ is also true and was already shown in \cite{HKT2008}.

\section{Preliminaries}

When making estimates, we often write the constants as positive
real numbers $C$ which may vary between
appearances, even within a chain of inequalities. These constants normally only depend on the dimension of the underlying space $\R^n$ unless otherwise stated.

 For any point $x \in \mathbb R^n$ and radius $r>0$ we denote the open ball by
 \[
 B(x,r) = \{y \in \mathbb R^n\,:\, |x-y| < r\}.
 \]
 More generally, for a set $A \subset \mathbb R^n$ we define the open $r$-neighbourhood as
 \[
 B(A,r) = \bigcup_{x \in A} B(x,r).
 \]
 
 We denote by $|E|$ the $n$-dimensional outer Lebesgue measure of a set $E \subset \mathbb R^n$.
 For any Lebesgue measurable subset $E \subset \mathbb R^n$ and any point $x \in \mathbb R^n$ we then define the upper density of $E$ at $x$ as
 \[
 \overline{D}(E,x) = \limsup_{r\searrow 0}\frac{|E\cap B(x,r)|}{|B(x,r)|},
 \]
 and the lower density of $E$ at $x$ as
 \[
 \underline{D}(E,x) = \liminf_{r\searrow 0}\frac{|E\cap B(x,r)|}{|B(x,r)|}.
 \]
 If  $\overline{D}(E,x) = \underline{D}(E,x)$, we call the common value the density of $E$ at $x$ and denote it by $D(E,x)$.
 The essential interior of $E$ is then defined as
 \[
 \mathring{E}^M = \{x \in \mathbb R^n \,:\, D(E,x)=1 \},
 \]
 the essential closure of $E$ as
 \[
 \overline{E}^M = \{x \in \mathbb R^n \,:\, \overline{D}(E,x) > 0\},
 \]
 and the essential boundary of $E$ as
 \[
 \partial^M E = \{x \in \mathbb R^n\,:\, \overline{D}(E,x) > 0\text{ and } \overline{D}(\mathbb R^n \setminus E,x) > 0\}.
 \]
As usual,  $\mathcal{H}^s(A)$ will stand for the $s$-dimensional Hausdorff measure of a set $A \subset \R^n$ obtained as the limit
\[
\mathcal H^s(A) = \lim_{\delta \searrow 0}\mathcal H_\delta^s(A),
\]
where $\mathcal H_\delta^s(A)$ is the $s$-dimensional Hausdorff $\delta$-content of $A$ defined as
\[
\mathcal H_\delta^s(A) = \inf\left\{\sum_{i=1}^\infty \diam(U_i)^s \,:\, A \subset \bigcup_{i=1}^\infty U_i, \diam(U_i) \le \delta\right\}.
\]

We say that a set $H\subset\R^n$ is $m$-rectifiable, for some $m<n$, if there exist countably many Lipschitz maps $f_j\colon\R^m\to \R^n$ so that $\mathcal{H}^m(H\setminus \bigcup_{j}f_j(\R^m))=0$. A set $H$ will be called  purely $m$-unrectifiable if for every Lipschitz map $f\colon \R^m\to\R^n$ we have 
$$\mathcal{H}^m(H\cap f(\R^m))=0.$$ 
Observe that by Rademacher's theorem one can deduce that if $f \colon \R^m\to\R^n$
is Lipschitz, then there are countably many sets $E_i\subset\R^m$ on which $f$ is bi-Lipschitz and such that
$\mathcal{H}^m(f(\R^m \setminus\bigcup_i E_i)) = 0$. 

Moreover, it easily follows that if $H \subset \mathbb R^n$ is not $m$-purely unrectifiable, then there exists a Lipschitz map $f \colon \mathbb R^m \to \mathbb R^{n-m}$ so that up to a rotation, the set
\[
H \cap \textrm{Graph}(f)
\]
has positive $\mathcal H^{m}$-measure, where $\textrm{Graph}(f) = \{(x,f(x))\,:\, x \in \mathbb R^m\}$.

 By a dyadic cube we refer to $Q = [0,2^{-k}]^n + \mathtt {j} \subset \R^n$ for some $k \in \mathbb Z$ and $\mathtt{j} \in 2^{-k}{\mathbb Z}^n$. We denote the side-length of such dyadic cube $Q$ by $\ell(Q) := 2^{-k}$.
 
 \subsection{$BV$-functions and sets of finite perimeter}
 
 Let us recall some basic results related to $BV$-functions and sets of finite perimeter.
For a more detailed account, we refer to the books \cite{AFP2000,EG2015,F1969}.
  
Differently to this paper, Mazy'a and Burago'  (see \cite{BM1967} and also \cite[Section 9.3]{mazya}) considered the space
$$BV_l(\Omega)=\{u\in L^{1}_{loc}(\Omega):\, \|Du\|(\Omega)<\infty\} $$
equipped with the seminorm $\|Du\|(\Omega)$. This way they defined $BV_l$-extension domains to be those $\Omega\subset\R^n$ for which just the total variation of the extension is controlled, that is, whenever $\|D(Tu)\|(\R^n)\leq C\|Du\|(\Omega)$. As we already explained in the introduction, they proved that being a $BV_l$-extension domain was equivalent to the fact that any set $E\subset\Omega$ of finite perimeter in $\Omega$ admits an extension $\widetilde E\subset\R^n$ satisfying only (PE1) and (PE2) from Definition \ref{def:strongperi}.  Note however, that thanks to \cite[Lemma 2.1]{KMS2010} $BV_l$-extension domains are equivalent to $BV$ extension domains if $\Omega$ is bounded.  

When working with $BV$ functions we will make use of the well-known $(1,1)$-Poincar\'e inequality that we now state (see for instance \cite[Theorem 3.44]{AFP2000} for the proof).
\begin{theorem}\label{thm:Poincare}
Let $\Omega\subset\R^n$ be an open bounded set with Lipschitz boundary. Then there exists a constant $C>0$ depending only on $n$ and $\Omega$ so that for every $u\in BV(\Omega)$ we have
$$\int_\Omega |u(y)-u_\Omega|\,dy\leq C \|Du\|(\Omega) .$$
In particular, there exists a constant $C>0$ only depending on $n$ so that if $Q,Q'\subset\R^n$ are two dyadic cubes  with $\frac{1}{4}\ell(Q')\leq \ell(Q)\leq 4\ell(Q')$ and  $\Omega=\text{int}(Q\cup Q')$ connected, then for every $u\in BV(\Omega)$,
\begin{equation}\label{Poincare}
\int_{\Omega} |u(y)-u_\Omega|\,dy\leq C \ell(Q)\|Du\|(\Omega) .
\end{equation}
\end{theorem}
We are using here the notation of the mean value integral of a function $u$ on the set $\Omega$ as
$$u_\Omega=\frac{1}{|\Omega|}\int_\Omega u(y)\,dy .$$

 Let us record as well the coarea formula for $BV$ functions. See for example \cite[Section 5.5]{EG2015}.
 
 \begin{theorem}\label{thm:coarea}
 Given a function $u\in BV(\Omega)$, the superlevel sets $u_t=\{x\in\Omega:\, u(x)>t\}$ have finite perimeter in $\Omega$ for almost every $t\in\R$ and
 $$\|Du\|(F)=\int^{\infty}_{-\infty} P(u_t,F)\,dt $$
 for every Borel set $F\subset \Omega$.
 Conversely, if $u\in L^1(\Omega)$ and $\int^{\infty}_{-\infty} P(u_t,\Omega)\,dt<\infty $ then $u\in BV(\Omega)$.
\end{theorem}

An important result due to Federer \cite[Section 4..5.11]{F1969} tells us that a set $E$ has finite perimeter in $\Omega$ if and only if $ \mathcal{H}^{n-1}(\partial ^M E\cap\Omega)<\infty$.  Moreover, thanks to De Giorgi's pioneering work \cite{DG1955} we can understand the structure of the boundary of sets of finite perimeter even better. Namely, if $E$ has finite perimeter in $\Omega$ then for every subset $A\subset \Omega$,
 $$\|D\chi_E\|(A)=P(E,A)= \mathcal{H}^{n-1} ( \partial^M E\cap A)$$
 and if $E$ has finite perimeter in $\R^n$ then
 $$\partial^M E=F\cup \bigcup_{n\in\N} K_n  $$ 
 where $\mathcal{H}^{n-1}(F)=0$ and $K_n$ are compact subsets of $C^1$ hypersurfaces.  
Furthermore, for any set $E$ with finite perimeter we have 
$$\overline{D}(E,x)=\underline{D}(E,x)\in \{0,1/2,1\} $$
for $\mathcal{H}^{n-1}$-almost every $x\in\R^n$.
Moreover, $\mathcal{H}^{n-1}(\partial ^M E\setminus \{x:\, D(E,x)=1/2   \})=0$, and hence for $\mathcal{H}^{n-1}$-almost every $x\in\partial^M E$ we have 
$$D(E,x) = \frac12.$$

Let us finally recall some terminology and results from \cite{ACMM2001}.
 A Lebesgue measurable set $E \subset \mathbb R^n$ with $|E|>0$ is called decomposable if there exist two Lebesgue measurable sets $F,G \subset \mathbb R^n$ so that $|F|,|G|>0$, $E = F \cup G$, $F\cap G = \emptyset$, and
 \[
 P(E,\mathbb R^n) = P(F,\mathbb R^n) + P(G,\mathbb R^n).
 \]
 A set is called indecomposable if it is not decomposable. For example, any connected open set $E\subset\R^n$ with $\mathcal{H}^{n-1}(\partial ^M E)<\infty$ is indecomposable.

For any set $E\subset\R^n$ of finite perimeter  we can always find a unique countable family of disjoint indecomposable subsets $E_i\subset E$ so that $|E_i|>0$, $P(E,\R^n)=\sum_{i} P(E_i,\R^n)$ and, moreover,
$$ \mathcal{H}^{n-1}\left(  \mathring{E}^M\setminus \bigcup_{i}   \mathring{E}_i^M  \right)=0.$$
For a proof of this result we refer to \cite[Theorem 1]{ACMM2001}.

In the particular case of $\R^2$, thanks to \cite[Corollary 1]{ACMM2001} one can find a decomposition of sets of finite perimeter into indecomposable sets whose  boundaries are  rectifiable Jordan curves, except for a set of $\mathcal{H}^1$-measure zero. We will state this useful result in the last Section 5, in Theorem \ref{thm:planardecomposition}.

\subsection{Measure density condition for $BV$-extension domains}

Nowadays it is a well-known fact that all $W^{1,p}$-extension domains for $1\leq p<\infty$ satisfy the measure density condition (see \cite{HKT2008}). 
Although we do not need this in our proofs, we record here the fact that the same property holds for $BV$-extension domains. Let us remark that a measure density condition for planar $BV_l$-extension domains was proven in \cite[Lemma 2.10]{KMS2010}. However, the proof does not seem to extend to domains in $\R^n$. The method of proof we employ here follows the same lines as \cite{HKT2008} and can be adapted for $BV_l$-extension domains as well.

\begin{proposition}\label{meas.dens:BV}
Let $\Omega\subset\R^n$ be a $BV$-extension or a $W^{1,1}$-extension domain, then there exists a constant $c>0$, depending only on $n$ and on the operator norm, so that for every $x\in\overline{\Omega}$ and $r\in (0,1]$ we have
$$|B(x,r)\cap \Omega|\geq c r^n .$$
\end{proposition}
\begin{proof}
We will only make the proof for $BV$-extension domains. For $W^{1,1}$-extension domains one can use the results from \cite{HKT2008}, or the fact that $W^{1,1}$-extension domains are $BV$-extension domains. A proof of this fact can be found in \cite[Lemma 2.4]{KMS2010}. The reader will notice that the key point will be to apply the Sobolev embedding theorem, which is both valid for $W^{1,1}$ and $BV$ functions.

Let us denote $r_0 = r$.
By induction, we define for every $i \in \N$ the radius $r_i \in (0,r_{i-1})$ by the equality
\[
|\Omega\cap B(x,r_i)| = \frac12|\Omega\cap B(x,r_{i-1})| = 2^{-i}|\Omega\cap B(x,r_0)|.
\]
Since $x \in \overline\Omega$, we have that $r_i \searrow 0$ as $i\to \infty$.

For each $i \in\N$, consider the function $f_i:\Omega \to \mathbb{R}$
\[
 f_i(y) = \begin{cases}1, & \text{for }y \in B(x,r_i) \cap \Omega,\\
 \frac{r_{i-1}-|x-y|}{r_{i-1}-r_i}, & \text{for }y \in (B(x,r_{i-1})\setminus B(x,r_i)) \cap \Omega,\\
 0, & \text{otherwise}.
 \end{cases}
\]
Note that these functions belong to the class $W^{1,1}(\Omega)$, in particular they are $BV$ functions. We can estimate their $BV$-norms by
\begin{align*}
\Vert f_i\Vert_{BV(\Omega)} &=\Vert f_i\Vert_{L^1(\Omega)} +\|D\, f_i\|(\Omega)  =\int_{\Omega} |f|+  \int_{\Omega} |\nabla f_i|  \\
& \le  |B(x,r_{i-1})\cap \Omega| + |r_i - r_{i-1}|^{-1} |(B(x,r_{i-1})\setminus B(x,r_i))\cap \Omega)| \\
&\leq C|r_i - r_{i-1}|^{-1} 2^{-i}|\Omega\cap B(x,r_0)|.
\end{align*}
Call $1^*=\frac{n}{n-1}$ and denote by $T\colon BV(\Omega)\to BV(\R^n)$ the extension operator. By the Sobolev inequality for BV functions (see \cite[Theorem 5.10]{EG2015}) we know that
$$\Vert Tf_i\Vert_{L^{1^*}(\R^n)}\leq C \Vert D(Tf_i)\Vert(\R^n),$$
where $C>0$ depends only on the dimension $n$.
Hence we have the following chain of inequalities
$$\|f_i\|_{L^{1^*}(\Omega)}\leq \|Tf_i\|_{L^{1*}(\mathbb{R}^n)}\leq C\|D(Tf_i)\|(\mathbb{R}^n)\leq C\|T\|\,\|f_i\|_{BV(\Omega)}. $$
We also have
\begin{align*}
\int_{\Omega} \vert f_i(y)\vert^{1^*}\,dy &\geq |B(x,r_{i})\cap \Omega|=2^{-i}|B(x,r_0)\cap \Omega|,
\end{align*}
and therefore
\begin{align*}
2^{-i}|B(x,r_0)\cap \Omega|&\leq    C\|T\|^{1^*}\left(|r_i - r_{i-1}|^{-1} 2^{-i}|\Omega\cap B(x,r_0)|\right)^{1^*} .
\end{align*}
Consequently, 
\begin{align*}
r_{i-1}-r_i & \leq C\|T\|2^{i(1/1^{*}-1)}|\Omega\cap B(x,r_0)|^{1-1/1^{*}} \\
&= C\|T\|2^{-i/n}|\Omega\cap B(x,r_0)|^{1/n}.
\end{align*}
By summing up all these quantities we conclude that
$$ r =r_0= \sum_{i=1}^\infty(r_{i-1} - r_{i}) \le C\|T\|\sum_{i=1}^\infty 2^{-i/n}|\Omega\cap B(x,r)|^{1/n} = \frac{C\|T\|}{2^{1/n}-1}|\Omega\cap B(x,r)|^{1/n}. $$
This gives the claimed inequality.
 \end{proof}

\section{Equivalence of $W^{1,1}$-extension and strong $BV$-extension domains}

This section is devoted to the proof of Theorem \ref{thm:mainRn}. The idea in going from a strong $BV$-extension to a $W^{1,1}$-extension is to first extend the $W^{1,1}$-function from the domain as a $BV$-function to the whole space and then mollify it in the exterior of the domain. In the mollification process it is important to check that we do not change the function too much near the boundary.

 \subsection{Whitney smoothing operator}

In this subsection we prove existence of a suitable smoothing operator from 
BV to $W^{1,1}$.
For similar constructions we recommend to the reader to have a look at  \cite{BHS2002,HK1998,LLW2020}.

\begin{theorem}\label{thm:smoothing}
 Let $A \subset B \subset \mathbb R^n$ be open subsets. There exist a constant $C$ depending only on the dimension $n$ and a linear operator
\[
S_{B,A} \colon BV(B) \to \left\{u \in BV(B) \,:\, u|_{A}\in W^{1,1}(A) \right\} 
\]
so that for any $u \in BV(B)$ we have $S_{B,A}u|_{B\setminus A}=u$, 
\begin{equation}\label{eq:normineq}
\|S_{B,A}u\|_{BV(B)} \leq C \|u\|_{BV(B)},
\end{equation}
and
\begin{equation}\label{eq:boundaryzero}
\|D(S_{B,A}u-u)\|(\partial A)=0,
\end{equation}
where $S_{B,A}u-u$ is understood to be defined in the whole $\mathbb R^n$ via a zero-extension. Moreover,
 the operator $S_{B,A}$ is also bounded when acting from the space $BV_l(A)$ into the homogeneous Sobolev space $L^{1,1}(A)$.
\end{theorem}

Recall that $L^{1,1}(A)=\{u\in L^{1}_{loc}(A):\, \nabla u\in L^{1}(A)\}$ stands for the homogeneous Sobolev space endowed with the seminorm $\|u\|_{L^{1,1}(A)}=\|\nabla u\|_{L^{1}(A)}$.

Let us briefly explain how the operator $S_{B,A}$ is constructed. We first take a Whitney decomposition of the open set $A$ and a partition of unity based on it. The operator on a $BV$-function $u$ is then defined as the sum of $u$ restricted to the complement of $A$ and the average values of $u$ in each Whitney cube of $A$ times the associated partition function.
This way, we immediately have that the function $S_{B,A}$ is left unchanged in the complement of $A$, and that in $A$ it is smooth. The inequality 
\eqref{eq:normineq} will follow in a standard way from the Poincar\'e inequality for $BV$-functions, whereas for showing \eqref{eq:boundaryzero} we will show that the average difference between $u$ and $S_{B,A}u$ near $\mathcal H^{n-1}$-almost every boundary point of $A$ tends to zero as we get closer to the point.
\medskip

Let us now give the definition of the operator doing the smoothing part.
Suppose $A\subset \mathbb R^n$ is an open set, not equal to the entire space $\R^n$. Let $\mathcal W = \{Q_i\}_{i=1}^\infty$ be the standard \emph{Whitney decomposition} of $A$, by which we mean that it satisfies the following properties:
\begin{itemize}
    \item[(W1)] Each $Q_i$ is a dyadic cube inside $A$.
    \item[(W2)] $A=\bigcup_i Q_i$ and for every $i \ne j$ we have $\text{int}(Q_i) \cap \text{int}(Q_j) = \emptyset$.
    \item[(W3)] For every $i$ we have $\ell(Q_i) \le \dist(Q_i,\partial A)\le 4\sqrt{n} \ell(Q_i)$,
    \item[(W4)]  If $Q_i\cap Q_j \ne \emptyset$, we have $\frac{1}{4}\ell(Q_i) \le  \ell(Q_j)\le 4\ell(Q_i)$.
\end{itemize}
The reader can find a proof of the existence of such a dyadic decomposition of the set $A$ in \cite[Chapter VI]{stein}.

For a given set $A$ and its Whitney decomposition $\mathcal W$ we take a partition of unity $\{\psi_i\}_{i=1}^\infty$ so that for every $i$ we have $\psi_i \in C^{\infty}(\mathbb R^n)$, {$\text{spt}(\psi_i)=\{x\in\R^n:\,\psi_i(x)\neq 0\} \subset B(Q_i,\frac18\ell(Q_i))$,} $\psi_i \ge 0$, $|\nabla \psi_i| \le C \ell(Q_i)^{-1}$ with a constant $C$ depending only on $n$, and 
\[
\sum_{i=1}^\infty \psi_i = \chi_A.
\]
With the partition of unity we then define for any $u \in BV(A)$ a function
\begin{equation}\label{eq:SWdef}
 S_{\mathcal W}u = \sum_{i=1}^\infty u_{Q_i}\psi_i.
\end{equation}

Let us start by showing that $S_{\mathcal W}$ maps to $W^{1,1}(A)$ boundedly. Even though we could obviously equivalently use the BV norm also on the target, we prefer to write it as the $W^{1,1}$-norm in order to underline the spaces where the operator will be used.

\begin{lemma}\label{lma:BVtoSobo}
 Let $S_{\mathcal W}$ be the operator defined in \eqref{eq:SWdef}. Then for any $u \in BV(A)$ we have $S_{\mathcal W}u \in C^\infty(A)$ and $\|S_{\mathcal W}u\|_{W^{1,1}(A)} \le C\|u\|_{BV(A)}$ with a constant $C$ depending only on $n$.
\end{lemma}
\begin{proof}
 By (W2) and the fact that $\text{spt}(\psi_i) \subset B(Q_i,\frac18\ell(Q_i))$ for every $i$, we know that $\text{spt}(\psi_i) \cap \text{spt}(\psi_j) \ne \emptyset$ implies
 that $Q_i \cap Q_j \ne \emptyset$. Therefore, any point in $A$ has a neighbourhood where $S_{\mathcal W}u$ is defined as a sum of finitely many $C^\infty$-functions. Consequently, $S_{\mathcal W}u \in C^\infty(A)$.
  For the $L^1$-norm of the function we can estimate
  \[
  \|S_{\mathcal W}u\|_{L^1(A)} \le \sum_{i=1}^\infty \|u_{Q_i}\psi_i\|_{L^1(A)}
  = \sum_{i=1}^\infty |u_{Q_i}|\|\psi_i\|_{L^1(A)} \le \sum_{i=1}^\infty |u_{Q_i}|2^n\ell(Q_i)^n = 2^n \|u\|_{L^1(A)}.
  \]
  For the estimate on the $L^1$-norm of the gradient we start with an estimate via the $(1,1)$-Poincar\'e inequality  \eqref{Poincare}
  \begin{align*}
   \|\nabla(S_{\mathcal W}u)\|_{L^1(Q_i)} & \le \sum_{Q_j \cap Q_i \ne \emptyset}|u_{Q_i}-u_{Q_j}|\|\nabla \psi_j\|_{L^1(A)}\\
   &\le \sum_{Q_j \cap Q_i \ne \emptyset}|u_{Q_i}-u_{Q_j}|C\ell(Q_j)^{n-1}\\
   & \le C\sum_{Q_j \cap Q_i \ne \emptyset}\ell(Q_j)^{-1}\int_{Q_i \cup Q_j}|u_{Q_i}-u(y)| + |u(y)-u_{Q_j}|\,dy\\
    & \le C\sum_{Q_j \cap Q_i \ne \emptyset}\ell(Q_j)^{-1}\left(2\int_{Q_i \cup Q_j}|u_{Q_i\cup Q_j}-u(y)| +  2\int_{Q_i \cup Q_j}|u_{Q_i\cup Q_j}-u_{Q_j}|\,dy\right)\\ 
   & \le C\sum_{Q_j \cap Q_i \ne \emptyset}(\|Du\|(Q_i\cup Q_j)),
  \end{align*}
  which then gives, by summing over all $i$, and noticing that in the final double sum the sets $Q_i\cup Q_j$ have finite overlap with a constant depending only on $n$,
  \begin{equation}\label{bound.tot.var.}
  \begin{split}
  \|\nabla(S_{\mathcal W}u)\|_{L^1(A)} & = \sum_{i=1}^\infty \|\nabla(S_{\mathcal W}u)\|_{L^1(Q_i)} \\
 &  \le C \sum_{i=1}^\infty\sum_{Q_j \cap Q_i \ne \emptyset}(\|Du\|(Q_i\cup Q_j)) \\
 &\leq C \|Du\|(A).
  \end{split}
\end{equation}
This concludes the proof of the lemma.
\end{proof}

The next lemma gives the crucial boundary behaviour that will imply \eqref{eq:boundaryzero}.

\begin{lemma}
 For the operator $S_{\mathcal W}$ defined in \eqref{eq:SWdef} and for any $u \in BV(A)$ we have
\begin{equation}\label{eq:boundaryagain}
\lim_{r\searrow 0}\frac{1}{|B(x,r)|}\int_{B(x,r)\cap A}|S_{\mathcal W}u(y) - u(y)|\,d y =0
\end{equation}
for $\mathcal H^{n-1}$-almost every point $x \in \partial A$.
\end{lemma}
\begin{proof}
Suppose \eqref{eq:boundaryagain} fails on a set $F\subset \partial A$ with $\mathcal H^{n-1}(F) > 0$. Without loss of generality, we may assume $F$ compact.
By going to a subset of $F$ if needed, we may further assume that there exists a constant $\delta > 0$ so that
\begin{equation*}
\limsup_{r\searrow 0} \frac{1}{|B(x,r)|}\int_{B(x,r)\cap A}|S_{\mathcal W}u(y) - u(y)|\,d y > \delta
\end{equation*}
for every $x \in F$.

Let $\varepsilon>0$. By the $5r$-covering lemma there exists a disjointed countable collection $\{B(x_i,r_i)\}_{i \in I}$ so that $x_i\in F$,
  $r_i<\varepsilon$ for all $i$,
\begin{equation}\label{eq:delta}
|B(x_i,r_i)|\leq \frac{1}{\delta}\int_{B(x_i,r_i)\cap A}|S_{\mathcal W}u(y) - u(y)|\,d y
\end{equation}
and
\[
F \subset \bigcup_{i\in I} B(x_i,5r_i).
\]
Similarly as in the proof of Lemma \ref{lma:BVtoSobo}, we first estimate in a Whitney cube $Q \in \mathcal W$ using the (1,1)-Poincar\'e inequality \eqref{Poincare}
\begin{equation}\label{eq:poincareasalways}
\begin{split}
    \int_{Q}|S_{\mathcal W}u(y) - u(y)|\,d y 
    & =  \int_{Q}\left|\sum_{Q_i \cap Q \ne \emptyset}(u_{Q_i}\psi_i(y) - u(y)\psi_i(y))\right|\,d y\\
    &\le \sum_{Q_i \cap Q \ne \emptyset} \int_Q|u_{Q_i} - u(y)|\,d y\\
    &\le \sum_{Q_i \cap Q \ne \emptyset} \int_{Q\cup Q_i}|u_{Q_i} - u(y)|\,d y\\
    & \le C \ell(Q) \sum_{Q_i \cap Q \ne \emptyset} \|Du\|(Q\cup Q_i).
\end{split}
\end{equation}
By the property (W3) of the Whitney decomposition, we conclude that if $Q \in \mathcal W$ is such that $Q \cap B(x_i,r_i) \ne \emptyset$, we have 
\[
\ell(Q) \le  \dist(Q,\partial A) \le \dist(Q,x_i) < r_i,
\]
and hence 
\[
Q \subset B(x_i,(\sqrt{n}+1)r_i) \subset B(F,(\sqrt{n}+1)\varepsilon).
\]
Similarly, for the same $Q$, if $Q_i \cap Q \ne \emptyset$ for some $Q_i \in \mathcal W$, by (W4), we get
\[
\ell(Q_i) \le 4\ell(Q),
\]
and so
\[
Q_i \subset B(x_i,(5\sqrt{n}+1)r_i) \subset B(F,(5\sqrt{n}+1)\varepsilon).
\]
Now, using the definition of the Hausdorff content, the inequality \eqref{eq:delta}, the estimate \eqref{eq:poincareasalways}, and the above consideration for the cubes $Q$, we get
\begin{align*}
\mathcal H_{5\epsilon}^{n-1}(F) & \le C \sum_{i\in I} r_i^{n-1}
\le C \sum_{i \in I} \frac{|B(x_i,r_i)|}{r_i}\\
& \le C \sum_{i \in I} \frac{1}{\delta r_i}\int_{B(x_i,r_i)\cap A}|S_{\mathcal W}u(y) - u(y)|\,d y\\
& \le C \sum_{Q\in\mathcal{W}}\sum_{i \in I} \frac{1}{\delta \ell(Q)}\frac{\ell(Q)}{r_i}\int_{Q}\chi_{B(x_i,r_i)}(y)|S_{\mathcal W}u(y) - u(y)|\,d y\\
& \le \frac{C}{\delta} \left(\sum_{Q \cap B(F,(\sqrt{n}+1)\varepsilon)\ne \emptyset}
\frac{1}{\ell(Q)}\int_{Q}|S_{\mathcal W}u(y) - u(y)|\,d y\right)\\
& \le \frac{C}{\delta} \left(\sum_{Q \cap B(F,(\sqrt{n}+1)\varepsilon)\ne \emptyset}
\sum_{Q_i \cap Q \ne \emptyset} \|Du\|(Q\cup Q_i)\right)\\
& \le \frac{C}{\delta}\|Du\|(B(F,(5\sqrt{n}+1)\epsilon) \cap A) \searrow 0
\end{align*}
as $\varepsilon \searrow 0$. 
Thus
\[
\mathcal H^{n-1}(F) = 0,
\]
giving a contradiction and concluding the proof.
\end{proof}

With the previous two lemmas we can now prove the main theorem of the section.

\begin{proof}[Proof of Theorem \ref{thm:smoothing}]

Let $S_{\mathcal W}$ be the operator defined in \eqref{eq:SWdef} and suppose that $u \in BV(B)$ is given. We define
\[
S_{B,A}u = u|_{B\setminus A} + S_{\mathcal W}u|_A.
\]

Consider $S_{B,A}u-u\in L^1(B)$ for which, by \eqref{eq:boundaryagain}, we have
\begin{equation}\label{prop.zero.ext.}
\lim_{r\searrow 0}\frac{1}{|B(x,r)|}\int_{B(x,r)\cap A}|S_{B,A}u(y)-u(y)|\,dy =0
\end{equation}
for $\mathcal{H}^{n-1}$-almost every $x \in \partial A$. Observe that $S_{B,A}u-u=0$ on $B\setminus A$.

Let us introduce the superlevel sets $E_t=\{y\in \mathbb R^n:\, S_{B,A}u(y)-u(y)>t\}$ for every $t\in\R$, where $S_{B,A}u-u$ is defined in the whole $\mathbb R^n$ via a zero-extension. We want to show that $\mathcal H^{n-1}(\partial^M E_t\cap \partial A)=0$ for almost every $t\in\R$ and the equality \eqref{eq:boundaryzero} will follow by a simple application of the coarea formula.  We proceed as follows.

In the case that $t<0$, observe that for every $y\in A\setminus E_t$ we have $|S_{B,A}u(y)-u(y)|\geq |t|$, then for $\mathcal H^{n-1}$-almost all $x\in\partial A$, by \eqref{prop.zero.ext.},
\begin{align*}
    \overline{D}(A\setminus E_t,x) & = \limsup_{r\searrow 0}\frac{|A\setminus E_t\cap B(x,r)|}{|B(x,r)|} \\
    & \leq  \limsup_{r\searrow 0}\frac{1}{|t||B(x,r)|}\int_{A\cap B(x,r)} |S_{B,A}u(y)-u(y)|\,dy =0 .
\end{align*}
This, together with the fact that $B\setminus A\subset E_t $, means that the set $E_t$ has density $1$ at $\mathcal H^{n-1}$-almost all points $x\in \partial A$.

If we take $t>0$, for every $y\in E_t$ we have $|S_{B,A}u(y)-u(y)|\geq t$, and then  for $\mathcal H^{n-1}$-almost all $x\in\partial A$, again by \eqref{prop.zero.ext.},
\begin{align*}\overline{D}( E_t,x) & = \limsup_{r\searrow 0}\frac{| E_t\cap B(x,r)|}{|B(x,r)|} \\
& \leq  \limsup_{r\searrow 0}\frac{1}{t |B(x,r)|}\int_{A\cap B(x,r)} |S_{B,A}u(y)-u(y)|\,dy =0 .
\end{align*}
This means, using  $E_t\subset A$, that the set $E_t$ has density $0$ at $\mathcal H^{n-1}$-almost all points $x\in\partial A$.

From these previous observations we  deduce that $\mathcal H^{n-1}(\partial^M E_t\cap \partial A)=0$ for all $t\neq 0$. 
We therefore  obtain \eqref{eq:boundaryzero}, applying the coarea formula,
$$\|D(S_{B,A}u-u)\|(\partial A)=\int^{\infty}_{-\infty} \mathcal{H}^{n-1}(\partial ^M E_t\cap \partial A)\,dt=0. $$

We now combine this with Lemma \ref{lma:BVtoSobo} to obtain \eqref{eq:normineq} and hence also that $S_{B,A}u \in BV(B)$. We get
\begin{align*}
\|D(S_{B,A}u)\|(B) & \le \|Du\|(B) + \|D(S_{B,A}u)-u\|(B)\\
& = \|Du\|(B) + \|D(S_{B,A}u)-u\|(A)\\
& \le \|Du\|(B) + \|D(S_{\mathcal W}u|_A)\|(A) + \|Du\|(A)\\
& \le \|Du\|(B) + C\|Du|_A\|(A) + \|Du\|(A)\\
& \le (C+2)\|Du\|(B)
\end{align*}
and conclude the proof.
 \end{proof}

\subsection{Proof of Theorem \ref{thm:mainRn}}

In this section we will prove Theorem \ref{thm:mainRn} with the help of Theorem \ref{thm:smoothing}.
Recall that we are claiming that for a bounded domain $\Omega\subset\R^n$ the following are equivalent:
\begin{enumerate}
    \item $\Omega$ is a $W^{1,1}$-extension domain.
    \item $\Omega$ is a strong $BV$-extension domain.
    \item $\Omega$ has the strong extension property for sets of finite perimeter.
\end{enumerate}

We will show the equivalence by showing the implications
\[
\text{(1)} \Longrightarrow \text{(3)} \Longrightarrow \text{(2)} \Longrightarrow \text{(1)}.
\]

\begin{proof}[Proof of the implication (1) $\Longrightarrow$ (3)] We start with the assumption that $\Omega$ is a  bounded $W^{1,1}$-extension domain. In particular, it is known that $\Omega$ is also a $L^{1,1}$-extension domain (see \cite{K1990}). That is, there exists an extension operator  $T\colon L^{1,1}(\Omega) \to L^{1,1}(\mathbb R^n)$  with $\|\nabla (Tu)\|_{L^1(\R^n)}\leq \|T\| \|\nabla u\|_{L^1(\Omega)}$ for every $u\in L^{1,1}(\Omega)$. Since $\Omega$ is bounded, after multiplying with a suitable Lipschitz cutoff-function we may assume that $Tu\in L^1(\R^n)$ and still keep the control on the gradient norm.

We claim that $\Omega$ has the strong extension property for sets of finite perimeter. Thus, let $E \subset \Omega$ be a set of finite perimeter in $\Omega$. We need to find a set $\widetilde E \subset \mathbb R^n$ so that (PE1)--(PE3) of Definition \ref{def:strongperi} hold.

Towards this, let $S_{\Omega,\Omega} \colon BV(\Omega) \to W^{1,1}(\Omega)$ be the operator given by 
Theorem \ref{thm:smoothing}.
We now define a function $v \in W^{1,1}(\mathbb R^n)$ by 
\[
v = TS_{\Omega,\Omega}\chi_E.
\]
By truncating the function if needed, we may assume that $0 \le v\le 1$.

 Applying the  coarea formula (Theorem \ref{thm:coarea}) for the function $v$, 
\[
 \int_{0}^1 P(\{v>t\},\mathbb R^n)\, dt= \|Dv\|({\mathbb R^n}) = \int_{\R^n} |\nabla v(y)|\,dy<\infty .
\]
This gives, in particular, that there exists a set $I\subset [0,1]$ 
with $\mathcal H^1(I) \ge \frac12$ for which for every $t \in I$
we have
\begin{equation}\label{control.per.ext.}
\begin{split}
P(\{v>t\},\mathbb R^n) &\le 2\|Dv\|({\mathbb R^n})=2\int_{\R^n} |\nabla v(y)|\,dy \leq 2\|T\|\, \|\nabla S_{\Omega,\Omega} \chi_E\|_{L^{1}(\Omega)}\\
& \leq 2 \|T\|\, C\|D\chi_E\|(\Omega)=2C\|T\|\, P(E,\Omega).
\end{split}
\end{equation}
In the penultimate inequality we are using \eqref{bound.tot.var.}.

By the measure density (Proposition \ref{meas.dens:BV}) we have $|\partial\Omega|=0$.
This together with the fact that $\nabla v \in L^{1}(\mathbb R^n)$, gives
\[
\|D v\|(\partial\Omega) = 0,
\]
and by  \eqref{eq:boundaryzero} 
\[
 \|D(\chi_{E} - v\chi_{\Omega})\|(\partial\Omega) = \|D(\chi_{E} - S_{\Omega,\Omega}\chi_{E})\|(\partial\Omega) =0,
\]
where $\chi_{E} - S_{\Omega,\Omega}\chi_{E}$ is understood to be defined in the whole $\mathbb R^n$ via a zero-extension.
Hence, again by the coarea formula 
\begin{align*}
\int_0^1\mathcal{H}^{n-1}(\partial^M(E \cup (\{v>t\}\setminus \Omega))\cap \partial\Omega)\,dt
& = \|D(\chi_{E} + v\chi_{\mathbb R^n\setminus\Omega})\|(\partial\Omega)\\
&\le \|D v\|(\partial\Omega) + \|D(\chi_{E} - v\chi_{\Omega})\|(\partial\Omega) = 0.
\end{align*}
This gives that for almost every $t \in [0,1]$ we have
\begin{equation}\label{eq:boundaryintersectionzero}
 \mathcal{H}^{n-1}(\partial^M(E \cup (\{v>t\}\setminus \Omega))\cap \partial\Omega) = 0.
\end{equation}

Let us pick $t \in I\subset [0,1]$ so that both \eqref{control.per.ext.} and
\eqref{eq:boundaryintersectionzero} hold, and define
\[
\widetilde E = E \cup (\{v>t\}\setminus \Omega).
\]

Now, it is straightforward that condition (PE1) holds. The equation \eqref{eq:boundaryintersectionzero} gives (PE3), and together with \eqref{control.per.ext.} it also implies 
\begin{align*}
 P(\widetilde E,\mathbb R^n) &= \mathcal H^{n-1}(\partial^M\widetilde E)\\
 &\le \mathcal H^{n-1}((\partial^M E)\cap \Omega)
 +\mathcal H^{n-1}((\partial^M\widetilde E) \cap \partial\Omega)
 +\mathcal H^{n-1}(\partial^M\{v > t\})\\
 &\le
 P(E,\Omega) + 2C\|T\|P(E,\Omega)\end{align*}
proving (PE2).
\end{proof}

\begin{proof}[Proof of the implication (3) $\Longrightarrow$ (2)]
{By assumption $\Omega$ has the strong extension property for sets of finite perimeter, so there exists a constant $C>0$ such that for any set $E\subset\Omega$ of finite perimeter there exists a set $\widetilde E\subset \R^n$ such that (PE1)--(PE3) are satisfied.}

Take a function $u\in BV(\Omega)$ and let $B \supset \overline{\Omega}$ be a large enough ball.
Without loss of generality, we may assume that $u \colon \Omega \to [0,1]$. Let us write $E_t = \{u \ge t\}=\{y\in\Omega:\, u(y)\geq t\}$ for the superlevel sets for each $t$. Since $u \in BV(\Omega)$, by the coarea formula, $P(E_t,\Omega) < \infty$ for almost every $t \in [0,1]$. For these $t$, we select 
$\tilde E_t$ to be a strong perimeter extension of $E_t$.
For convenience, for the remaining $t$ we define $\tilde E_t = E_t$. Notice that these are not strong perimeter extensions of $E_t$. This will not pose a problem for us, since we will not use these values of $t$ in the construction below.

Before going to the actual proof, let us note that
if the strong perimeter extensions $\tilde E_t$ could be chosen so that $(t,x) \mapsto \chi_{\tilde E_t}(x)$ is measurable, by Fubini's theorem we would obtain
\[
\|DTu\| \le \int_0^1 P(\tilde E_t)\,dt
\]
for the function $Tu(x) = \mathcal H^1\{t \in [0,1]\,:\, x \in \tilde E_t\}$.
In order to circumvent the measurability issue, we proceed by defining the extension $Tu$ in a similar way, but as a limit of simple functions $u_m$.

For every $t \in [0,1]$, let us denote by $I_k(t)$ the (half-open) dyadic interval of length $2^{-k}$ containing $t$. 
For almost every $t \in [0,1]$ we then have
\begin{equation}\label{eq:tdensity}
P(E_t,\Omega) \le \limsup_{k \to \infty} 2^k\int_{I_k(t)}P(E_s,\Omega)\,ds.
\end{equation}
For almost every $t \in [0,1]$ we also have 
\[
P(\tilde E_t,\partial\Omega)= 0.
\]
Since $\partial\Omega$ is a compact set, for almost every $t$ we then have
\[
\lim_{k\to \infty}P(\tilde E_t,B(\partial\Omega,2^{-k}))= 0.
\]
Let us write for each $k,m \in \mathbb N$
\[
I_k^m = \{t \in [0,1] \,:\, P(\tilde E_t,B(\partial\Omega,2^{-k})) < 2^{-m} \text{ and \eqref{eq:tdensity} holds}\}.
\]
Notice that the sets $I_k^m$ are not necessarily measurable.
Nevertheless, since $\mathcal H^1$ is a regular outer measure, we have
\[
\mathcal H^1(I_k^m) \nearrow 1, \qquad \text{as }k \to \infty
\]
for every $m \in \mathbb N$.

We define a sequence $(k_m)_{m =1}^\infty \subset \mathbb N $ inductively as follows. First take $k_1 \in \mathbb N$ so that
\[
\mathcal H^1(I_{k_1}^1) > 1- 2^{-1}.
\]
Suppose now that $k_i$ has been defined for all $i < m$. Then we take $k_m \in \mathbb N$ so that
\begin{equation}\label{eq:smallcomplements}
\mathcal H^1\left(\bigcap_{j=i}^mI_{k_j}^j\right) > 1- 2^{-i}
\end{equation}
for all $i \le m$. Notice that this requirement can be obtained 
since \eqref{eq:smallcomplements} is with a strict inequality 
and again by outer regularity, for every $i < m$ we have
\[
\mathcal H^1\left(I_{k}^m \cap \bigcap_{j=i}^{m-1}I_{k_j}^j\right) \to
\mathcal H^1\left(\bigcap_{j=i}^{m-1}I_{k_j}^j\right)
\]
as $k \to \infty$.

Now, for $m \in \mathbb N$ we also take $l_m \in \mathbb N$ for which 
\begin{equation}\label{eq:intervalscover}
    \mathcal{H}^1 (J^m) > 1 - 2^{-m}, 
\end{equation} 
 where
\[
J^m = \left\{t \in [0,1]\,:\, P(E_t,\Omega) \le  2^{l_m+1}\int_{I_{l_m}(t)}P(E_s,\Omega)\,ds\right\}
\]
The index $l_m$ then gives us the scale at which the simple function $u_m$ is constructed.

Let us now construct the function $u_m$ for a given $m \in \mathbb N$.
For each $j \in \{1,\dots, 2^{l_m}\}$ define 
\[
i_j^m = \min\left\{i \,:\, [(j-1)2^{-l_m},j2^{-l_m}) \cap J^m \cap \bigcap_{h=i}^mI_{k_h}^h \ne \emptyset \right\}.
\]
Notice that always $i_j^m\le m+1$ since $[(j-1)2^{-l_m},j2^{-l_m}) \cap J^m \ne \emptyset$.

We then select 
\[
t_j^m \in [(j-1)2^{-l_m},j2^{-l_m}) \cap J^m \cap \bigcap_{h=i_j^m}^mI_{k_h}^h. 
\]

Next, we define
\[
u_m = \sum_{j=1}^{2^{l_m}}2^{-l_m}\chi_{\tilde E_{t_j^m}},
\]
which satisfies $0\leq u_m\leq 1$ and  $u_m\in BV(B)$.  

For every $i \in \{1,\dots,m\}$, let us denote
\[
K_i^m = \left\{j \in \{1, \dots, 2^{l_m}\}\,:\, [(j-1)2^{-l_m},j2^{-l_m}) \cap J^m \cap \bigcap_{h=i}^mI_{k_h}^h = \emptyset \right\}
\]
and
\[
B_i^m = \bigcup_{j \in K_i^m} [(j-1)2^{-l_m},j2^{-l_m}).
\]
Since $J^m$ is measurable, \eqref{eq:intervalscover} gives
$\mathcal H^1\left([0,1] \setminus J_m\right) < 2^{-m}$,
and thus, by \eqref{eq:smallcomplements}
\begin{align*}
\mathcal H^1\left(\bigcap_{h=i}^mI_{k_h}^h \cap J_m\right)
& = \mathcal H^1\left(\bigcap_{h=i}^mI_{k_h}^h \right) - \mathcal H^1\left(\bigcap_{h=i}^mI_{k_h}^h \setminus J_m\right)\\
& \ge \mathcal H^1\left(\bigcap_{h=i}^mI_{k_h}^h \right) - \mathcal H^1\left([0,1] \setminus J_m\right)\\
& > 1 - 2^{-i} - 2^{-m} \ge 1 - 2^{-i+1}.
\end{align*}
Hence, we have
\begin{equation}\label{eq:badestimate}
\mathcal H^1(B_i^m) < 2^{-i+1}.
\end{equation}

For the norm of $Du_m$, by the fact that $t_j^m \in J^m$ for every $j$, we get the estimate
\begin{align*}
\|Du_m\|(\mathbb R^n) &\le \sum_{j=1}^{2^{l_m}}2^{-l_m} P(\tilde E_{t_j^m}, \mathbb R^n) \le C \sum_{j=1}^{2^{l_m}}2^{-l_m} P(E_{t_j^m}, \Omega)\\
&\le C \sum_{j=1}^{2^{l_m}}2\int_{[(j-1)2^{-l_m},j2^{-l_m})} P(E_{s},\Omega)
= 2C\|Du\|(\Omega).
\end{align*}
Hence, there exists a subsequence of $(u_m)_{m=1}^\infty$, which converges in $L^1(B)$ to a function $v\in BV(B)$. For it, we have
\[
\|Dv\|(B) \le \limsup_{m \to \infty}\|Du_m\|(\mathbb R^n) \le 2C\|Du\|(\Omega).
\]
Moreover, clearly $v = u$ on $\Omega$.

In order to estimate $\|Dv\|(\partial\Omega)$
we observe that, for every $i \in \{1,\dots,m\}$, we have, by \eqref{eq:badestimate},
\begin{equation}\label{eq:nearboundarysmall}
\begin{split}
 \|Du_m\|(B(\partial\Omega,2^{-k_i}))
 & \le \sum_{j=1}^{2^{l_m}}2^{-l_m} P(\tilde E_{t_j}, B(\partial\Omega,2^{-k_i}))\\
 &\le \sum_{j\notin K_i^m}2^{-l_m} P(\tilde E_{t_j}, B(\partial\Omega,2^{-k_i})) + 
 \sum_{j\in K_i^m}2^{-l_m} P(\tilde E_{t_j}, \mathbb R^n)\\
 & \le 2^{-i} + C2\int_{B_i^m}P(E_s,\Omega)\,ds\\
 & \le 2^{-i} + C2\delta(2^{-i+1}),
 \end{split}
\end{equation}
where 
\[
\delta(r) = \sup\left\{\int_{A}P(E_s,\Omega)\,ds\,:\,A \subset [0,1], \mathcal H^1(A) = r \right\} \to 0
\]
as $r  \to 0$ by the absolute continuity of the integral.
Since the upper bound in \eqref{eq:nearboundarysmall} goes to zero as $i \to \infty$ independently of $m$, we have
\[
\|Dv\|(\partial\Omega) = 0.
\]
Let us assume that the function $v$ is extended as zero outside $B$. Recall that we have
$$\|Dv\|(B)\leq 2C \|Du\|(\Omega) \;\;\text{and}\;\; \|Dv\|(\partial \Omega)=0.  $$ 
In order to conclude the proof we control the $BV$-norm of the function $v$ in the whole $\R^n$ as follows.

Consider a Lipschitz function $\eta$ which takes the value $1$ on $\Omega$ and has support in $B$. Then one can check that 
$$ \|D(\eta v)\|(\R^n) \leq C \|Dv\|(B) $$ and using the Poincar\'e inequality that
$$\|\eta v\|_{L^1(\R^n)}=\| \eta v\|_{L^1(B)}\leq C\|D(\eta v)\|(B)\leq C \|Dv\|(B) .$$
Therefore we have $\|\eta v\|_{BV(\R^n)}\leq C \|Du\|_{BV(\Omega)}$, where the constant $C$ depends on the constant coming from (PE2), on $|\Omega|$ and on the constant coming from the Poincar\'e inequality.
We then can assure that $T\colon  BV(\Omega)\to BV(\R^n) \colon u \mapsto \eta v$ is an extension operator.

Obviously we still have  $\|D(\eta v)\|(\partial\Omega)=\|Dv\|(\partial\Omega)=0$. 
  Hence $\Omega$ is indeed a strong $BV$-extension domain.
\end{proof}

\begin{proof}[Proof of the implication (2) $\Longrightarrow$ (1)]

We start with a strong $BV$-extension operator $$T\colon BV(\Omega)\to BV(\R^n).$$
In particular, we know that 
\begin{equation}\label{Sing.part.=0}
\|D(Tu)\|(\partial \Omega)=0 
\end{equation}
for every $u\in BV(\Omega)$.

Let $S=S_{\R^n, (\R^n\setminus\overline \Omega)} $ be a Whitney smoothing operator given by Theorem \ref{thm:smoothing}. We assert that the operator $R \colon W^{1,1}(\Omega)\to W^{1,1}(\R^n)$ defined by $Ru(x)=(S\circ T)(u)(x)$ is a $W^{1,1}$-extension operator. 

Observe that $Ru=u$ on $\Omega$ and 
$$\|Ru\|_{BV(\R^n)}\leq C\|Tu\|_{BV(\R^n)}\leq C\|T|\|\, \|u\|_{BV(\Omega)}=C\|T\|\, \|u\|_{W^{1,1}(\Omega)}.$$
To conclude we must check that indeed $Ru\in W^{1,1}(\R^n)$, so that in particular $\|Ru\|_{BV(\R^n)}=\|Ru\|_{W^{1,1}(\R^n)}$. In order to get this let us show that the Radon measure $\|D(Ru)\|$ consists only of its absolutely continuous part, and not of its singular part. Since we already now that $Ru|_{\Omega}$ and $Ru|_{\R^n\setminus \overline\Omega}$ are $W^{1,1}$ functions we merely have to prove that $\|D(Ru)\|(\partial\Omega)=0$.  By the special properties of our smoothing operator given by \eqref{eq:boundaryzero} and by our assumption \eqref{Sing.part.=0}  we have  that $$\|D(Ru)\|(\partial \Omega)=\|D(STu -Tu+Tu)\|(\partial \Omega) \leq \|D(STu-Tu)\|(\partial\Omega)+\|D(Tu)\|(\partial\Omega)=0$$ 
and we are done.
\end{proof}

\section{Further properties of $W^{1,1}$-domains}\label{sec:corollaries}
In this section we prove some corollaries to Theorem \ref{thm:mainRn}.

\begin{corollary}\label{cor:density}
Let $\Omega \subset \mathbb R^n$ be a bounded $W^{1,1}$-extension domain. Then the set of points $x \in \partial \Omega$ with 
$\overline{D}(\Omega,x)>\frac12$ is purely $(n-1)$-unrectifiable.
\end{corollary}
\begin{proof}
 If the set
 \[
  F = \left\{x \in \partial\Omega \,:\, \overline{D}(\Omega,x)>\frac12\right\}
 \]
 is not purely $(n-1)$-unrectifiable, there exists a Lipschitz map $f \colon \mathbb R^{n-1} \to \mathbb R$ so that, after a suitable rotation,
\[
\mathcal H^{n-1}(\textrm{Graph}(f) \cap F) > 0.
\]
Notice that 
the set $\mathbb R^n \setminus \textrm{Graph}(f)$ consists of two connected components. Select one of the components that has nonempty intersection with $\Omega$ (actually, both have) and call $E$ its restriction to $\Omega$. Then 
\[
\partial^M E \cap \Omega = \textrm{Graph}(f) \cap \Omega
\]
and so in particular $E$ has finite perimeter in $\Omega$.
Let $\widetilde E \subset \mathbb R^n$ be any set of finite perimeter with $\widetilde E \cap \Omega = E$.
Since
\[
D(E,x) = \frac12
\]
at $\mathcal {H}^{n-1}$-almost every point  $x \in \partial^M E$, and $\overline{D}(\Omega,x)>\frac12$ for every $x \in F$, we have
\[
\mathcal H^{n-1}(F \cap \partial^M  E) = \mathcal H^{n-1}(F \cap \textrm{Graph}(f) ).
\]
Using again the fact that
\[
D(E,x) = \frac12
\]
at $\mathcal {H}^{n-1}$-almost every point  $x \in \partial^M E$, and $\underline{D}(\mathbb R^n \setminus \Omega,x)< \frac12$ for every $x \in F$, we have
\[
0 < \underline{D}(E,x) \le  \underline{D}(\widetilde E,x) \le \underline{D}(\mathbb R^n\setminus  \Omega,x) + \underline{D}(E,x) < 1
\]
for $\mathcal {H}^{n-1}$-almost every point $x \in F \cap \partial^ME$. This means that there exists a set $G \subset F\cap \partial^M E$ with $\mathcal{H}^{n-1}(G)=0$ for which
\[
(F\cap \partial^M E)\setminus G\subset F\cap  \partial^M \widetilde E.
\]
Consequently,
$$ \mathcal H^{n-1}(\partial\Omega \cap \partial^M \widetilde E)  \geq \mathcal H^{n-1}(F \cap \partial^M \widetilde E)\geq \mathcal H^{n-1}(F \cap \partial^M  E) = \mathcal H^{n-1}(F \cap \textrm{Graph}(f) )> 0.$$
Hence $\Omega$ does not have the strong extension property for sets of finite perimeter ((PE3) fails), and so by Theorem \ref{thm:mainRn} it is not a $W^{1,1}$-extension domain.
\end{proof}

The next example shows that even in the plane the conclusion of Corollary \ref{cor:density} is not sufficient to imply that a bounded $BV$-extension domain is a $W^{1,1}$-extension domain.

\begin{example}\label{ex:1}
Let us construct a planar $BV$-extension domain $\Omega$ so that the upper-density of $\Omega$ at all except at  countably many  boundary-points is at most $1/2$, but the domain is not a $W^{1,1}$-extension domain.
We set
\[
\Omega = (-1,1)^2 \setminus \left((\{0\}\times [-1/2,1/2]) \cup \bigcup_{i=2}^\infty E_i \right),
\]
where, for every $i \geq 2$, we define
\begin{align*}
 E_i = \bigcup_{k=0}^{2^i-1}&\bigg( [-2^{-i+1},-2^{-i}-2^{-i-10}] \cup [2^{-i}+2^{-i-10},2^{-i+1}]\\
 &\times[-2^{-1} + k2^{-i},-2^{-1} + (k+1)2^{-i}-2^{-i-10}]\bigg).
\end{align*}
See Figure \ref{fig:example.4.2} for an illustration.
\begin{figure}
     \centering
     \includegraphics[width=0.4\textwidth]{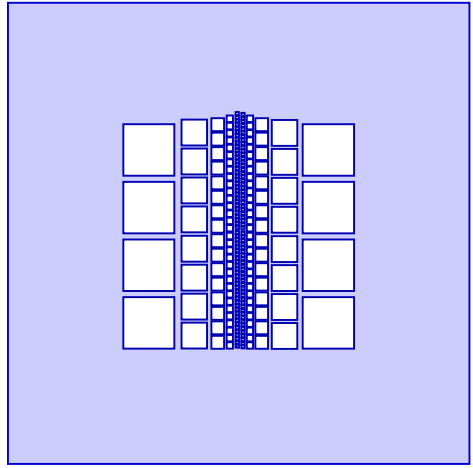}
     \psfrag{O}{$\Omega$}
     \caption{An illustration of the $BV$-extension domain $\Omega \subset \mathbb R^2$ in Example \ref{ex:1} which is not a $W^{1,1}$-extension domain. Components of the complement accumulate on the vertical line segment where the upper-density is less than $\frac12$ at almost every point.
     }
     \label{fig:example.4.2}
 \end{figure}
Now, the upper-density of $\Omega$ is clearly at most $1/2$ at all the points of the boundary $\partial \Omega$ except for the corners of the connected components of $E_i$, and the points $(0,-1/2)$ and $(0,1/2)$, which together form only a countable set. (One could remove balls instead of rectangles to get the upper-density bound for all boundary points.)

The domain $\Omega$ is a $BV$-extension domain because each removed square has a neighbourhood inside $\Omega$ from which the $BV$-function can be extended to the square with a uniform constant.
These neighbourhoods can be taken pairwise disjoint. This will result in an extension operator
\[
T \colon BV(\Omega) \to BV((-1,1)^2 \setminus \{0\}\times[-1/2,1/2]).
\]
The target set clearly admits an extension to $BV(\mathbb R^2)$.

The domain $\Omega$ is not a $W^{1,1}$-extension domain, because the set
$\{0\}\times [-1/2,1/2]$ is not purely $1$-unrectifiable, and this is the set $H$ in the following Corollary \ref{cor:unrectifiabilitynecessity}.
\end{example}

The next corollary to Theorem \ref{thm:mainRn} shows that one direction in Theorem \ref{thm:planar} holds also in higher dimensions.

\begin{corollary}\label{cor:unrectifiabilitynecessity}
Suppose that $\Omega \subset \mathbb R^n$ is a bounded $W^{1,1}$-extension domain. Let $\Omega_i$, for $i \in I$, be the connected components of $\mathbb R^n \setminus \overline{\Omega}$. Then the set 
\[
 H = \partial \Omega \setminus \bigcup_{i \in I}\overline{\Omega_i}
\]
is purely $(n-1)$-unrectifiable.
\end{corollary}
\begin{proof}
 Supposing $\Omega$ to be a $W^{1,1}$-extension domain, by Theorem \ref{thm:mainRn} we know that it has the strong perimeter extension property. 
 
 Now, towards a contradiction, suppose that $f \colon \mathbb R^{n-1} \to \mathbb R$ is an $L$-Lipschitz map so that 
 \[
 \mathcal H^{n-1}(\textrm{Graph}(f) \cap H) > 0,
 \]
 after a suitable rotation.
 Let $A$ be a component of $\mathbb R^n \setminus \textrm{Graph}(f)$
 such that the set 
 \[
 F = \{ x \in  \textrm{Graph}(f) \cap H \,:\,\overline D(\Omega \cap A, x)> 0\}
 \]
 has positive $\mathcal H^{n-1}$-measure.  By the measure density of $\Omega$ (Proposition \ref{meas.dens:BV}), at least one of the components must satisfy this.
 Without loss of generality, we may assume that 
 \[
 A = \{(y,f(x))\,:\, y < f(x)\}.
 \]

Take $E = \Omega \cap A$ and let $\widetilde E$ be the strong perimeter extension of $E$. Now, since $\mathcal H^{n-1}(\partial^M\widetilde E \cap \partial\Omega) = 0$, the set
\[
G = \{x \in F\,:\, D(\widetilde E, x) = 1\}
\]
has positive $\mathcal H^{n-1}$-measure. Take $x = (x_1,\dots,x_n) \in G$. Since $E \subset A$, which was bounded by a graph of an $L$-Lipschitz map, the set 
\[
 R_{x,L} = \{y = (y_1,\dots, y_n) \in \mathbb R^n \,:\, y_n - x_n>L|(x_1,\dots,x_{n-1}) - (y_1,\dots,y_{n-1})|\}
\]
does not intersect $E$. If there exists a small radius $r>0$ for which 
\[
R_{x,2L} \cap B(x,r) \cap \Omega = \emptyset,
\]
we conclude that there exists a connected component $\Omega_i$ of $\mathbb R^n \setminus \overline{\Omega}$ for which $x \in \partial\Omega_i$ contradicting the fact that $x \in H$. Hence, there exists a sequence of points $x^i \in R_{x,2L} \cap \Omega$ such that $|x^i-x| \to 0$. 
Since $f$ is $L$-Lipschitz, writing $\delta= \frac{1}{4(L+1)}$, we have
\[
B\left(x^i,\delta|x^i-x|\right) \subset R_{x,L} \subset \mathbb R^n \setminus E.
\]
By the measure density (Proposition \ref{meas.dens:BV}), we have 
\[
|B(x^i,\delta|x^i-x|)\cap \Omega| > c|B(x^i,\delta|x^i-x|)|
\]
for all $x_i$.
Thus,
\[
\frac{|B(x,2|x^i-x|)\cap \widetilde E|}{|B(x,2|x^i-x|)|} \le 
1 - \frac{|B(x^i,\delta|x^i-x|)\cap \Omega|}{|B(x,2|x^i-x|)|} < 1-c\left(\frac{\delta}{2}\right)^n < 1
\]
giving that
\[
\underline{D}(\widetilde E,x) < 1,
\]
which contradicts the fact that $x \in G$.
\end{proof}
Let us point out that if in addition we require $\Omega$ to be planar and simply connected in the previous corollary, we would get the stronger fact that $\mathcal{H}^1(\partial \Omega\setminus\bigcup_i \overline{\Omega}_i)=0$.

In the next section we will show that in the planar case the conclusion of Corollary \ref{cor:unrectifiabilitynecessity} is also a sufficient condition for a bounded $BV$-extension domain to be a $W^{1,1}$-extension domain. The following example shows that this is not the case in dimension three.

\begin{example}
Let us construct a bounded $BV$-extension domain $\Omega \subset \mathbb R^3$ which is not a $W^{1,1}$-extension domain so that $\mathbb R^3 \setminus \overline{\Omega}$ consists of only one component $\Omega_0$ for which $\partial \Omega = \partial \Omega_0$. Consequently, in the statement of Corollary \ref{cor:unrectifiabilitynecessity} we have $H = \emptyset$.

Let $C \subset [0,1]^2$ be a Cantor set with $\mathcal H^2(C)>0$ and let 
\[
\Omega = (-1,1)^3 \setminus \left\{(x_1,x_2,x_3)\,:\, |x_3| \le \dist((x_1,x_2),C), (x_1,x_2) \in [0,1]^2\right\}.
\]

The fact that $\mathbb R^3 \setminus \overline{\Omega}$ consists of only one component $\Omega_0$ for which $\partial \Omega = \partial \Omega_0$ is immediate from the construction.

Also, with the same arguments as in the previous two corollaries, we see that 
\[
E = \{(x_1,x_2,x_3) \in \Omega\,:\, x_3 < 0\}
\]
does not have a strong perimeter extension.

In order to see that $\Omega$ is a $BV$-extension domain, take $u\in BV(\Omega)$.  First notice that since the parts 
\[
\Omega_1=\{(x_1,x_2,x_3) \in \Omega \,:\, x_3>0\} \;\text{ and } \;\Omega_2=\{(x_1,x_2,x_3) \in \Omega \,:\, x_3<0\}
\]
have Lipschitz boundaries, similarly to  \cite[Theorem 5.8]{EG2015} we can consider the zero extension of both $u|_{\Omega_1}$ and $u|_{\Omega_2}$ to the whole $\R^3$ and calling them $\tilde u_1$ and $\tilde u_2$ respectively, we have $\tilde u_1,\tilde u_2\in BV(\R^3)$ with

\begin{equation}\label{ex:union.BV.func}
\|D \tilde u_i\|(\R^3)=\|D u\|(\Omega_i)+\int_{\partial \Omega_i} |\text{Tr}_i(u)|\,d\mathcal{H}^2
\end{equation}
for every $i=1,2$. Here
\begin{align*}
    \text{Tr}_i\colon BV( \Omega_i)\to L^1\left(\partial\Omega_i; \mathcal{H}^2\right)
\end{align*}
for $i=1,2$ are bounded linear operators, called the traces, which are defined as
\begin{align*}
    &\text{Tr}_i(u)(x) = \lim_{r\searrow  0}\frac{1}{|B(x,r)\cap \Omega_i|}\int_{B(x,r)\cap\Omega_i}u(y)\,dy
 \end{align*}
for $\mathcal{H}^2$-almost every $x$. Now it is easy to check, following \eqref{ex:union.BV.func},
\begin{align*}
\|D \tilde u_i\|(\R^3) & \leq\|D u\|(\Omega_i)+\int_{\partial\Omega_i} |\text{Tr}_i(u)|\,d\mathcal{H}^1\\
&\leq \|u\|_{BV(\Omega_i)}+ C\|u\|_{BV(\Omega_i)}=(1+C)\|u\|_{BV(\Omega)},    
\end{align*}
for $i=1,2$.
To conclude, we just let our extension operator $T\colon BV(\Omega)\to BV(\R^3)$  be $Tu=\tilde u_1+\tilde u_2$, 
which is the zero extension of $u$ outside $\Omega$.
\end{example}

In the case where $\overline{\Omega} = \mathbb R^n$, the study of extension domains $\Omega$ is the same as the study of closed removable sets.  Notice that by the measure density (Proposition \ref{meas.dens:BV}) the Lebesgue measure of $\partial \Omega$ is zero for a Sobolev or $BV$-extension domain.  We call a set $F \subset \mathbb R^n$ of Lebesgue measure zero a removable set for $BV$, if $BV(\mathbb R^n \setminus F) = BV(\mathbb R^n)$ as sets and $\|Du\|(\mathbb R^n) = \|Du\|(\mathbb R^n\setminus F)$ for every $u \in BV(\mathbb R^n)$. Similarly, we call $F$ removable for $W^{1,1}$, if $W^{1,1}(\mathbb R^n \setminus F)= W^{1,1}(\mathbb R^n)$.
We obtain the following equivalence of removability.

\begin{corollary}
Let $F \subset \mathbb R^n$ be a closed set of Lebesgue measure zero. Then $F$ is removable for $BV$ if and only if $F$ is removable for $W^{1,1}$.
\end{corollary}
\begin{proof}
 Suppose $F$ is removable for $BV$. Then $F$ is purely $(n-1)$-unrectifiable. Otherwise, similarly as in the proof of Corollary \ref{cor:unrectifiabilitynecessity}, we can construct a set $E$ of finite perimeter so that $\mathcal H^{n-1}(\partial^M E\cap F)>0$. Hence, $P(E,\mathbb R^n\setminus F) \ne P(E,\mathbb R^n)$, contradicting the assumption that $F$ is removable for $BV$.
 Now, since $F$ is removable for $BV$, for every radius $R>0$, the set $B(0,R) \setminus F$ is a $BV$-extension domain. Since $F$ is purely $(n-1)$-unrectifiable, $B(0,R) \setminus F$ trivially has the strong perimeter extension property and is thus a $W^{1,1}$-extension domain by Theorem \ref{thm:mainRn}. Consequently, $F$ is removable for $W^{1,1}$. 
 
 Suppose then that $F$ is removable for $W^{1,1}$. Let $u \in BV(\mathbb R^n\setminus F)$. We only need to check that the function $u$ when seeing as a function defined on the whole $\R^n$, satisfies $\|Du\|(F)=0$. With the Whitney smoothing operator $S_{\mathbb R^n \setminus F,\mathbb R^n \setminus F}$  from Theorem \ref{thm:smoothing} we can modify $u$ to be a $W^{1,1}$-function $\tilde u = S_{\mathbb R^n \setminus F,\mathbb R^n \setminus F}u$ on $\R^n\setminus F$ and moreover, by  \eqref{eq:boundaryzero}, 
 $$\| D (\tilde u -u)\|(F)=0,$$
 where $\tilde u$ can be defined as any value on $F$. Since $F$ is removable for $W^{1,1}$, we have $\tilde u \in  W^{1,1}(\mathbb R^n)$. Thus $\|D \tilde u\|(F)=0$ because $|F|=0$ and therefore
 $$\|D u\|(F)\leq \|D( u-\tilde u)\|(F)+\|D \tilde u\|(F)=0, $$
 and we get that $u \in BV(\mathbb R^n)$ with $\|Du\|(\mathbb R^n) = \|Du\|(\mathbb R^n \setminus F)$.
\end{proof}

\section{Characterization of planar $W^{1,1}$-extension domains}\label{sec:planar}
 
 In this section we prove Theorem \ref{thm:planar} using the higher dimensional result stated in Theorem \ref{thm:mainRn}. Since the necessity part of Theorem \ref{thm:planar} holds in the higher-dimensional case by Corollary \ref{cor:unrectifiabilitynecessity}, we only need to prove the sufficiency. We first set some notations and definitions.

We say that $\Gamma\subset\R^2$ is a Jordan curve if $\Gamma=\gamma([a,b])$ for some $a,b\in \R$, $a<b$, and some continuous map $\gamma$, injective on $[a,b)$ and such that $\gamma(a)=\gamma(b)$. Accordingly to the famous Jordan curve theorem any Jordan curve $\Gamma$ splits $\R^2\setminus \Gamma$ in exactly two connected components, a bounded one and an unbounded one that we call $\text{int} (\Gamma)$ and $\text{ext} (\Gamma)$ respectively. We will often talk about rectifiable Jordan curves $J$, for which we mean that $J$ is a Jordan curve and it is $1$-rectifiable. A set $A$ whose boundary $\partial A$ is a Jordan curve is called a Jordan domain.

For technical reasons we also add to the class of Jordan curves the formal ''Jordan'' curves $J_0$  and $J_{\infty}$, whose interiors are $\R^2$ and the empty set respectively and for which we set $\mathcal{H}^{1}(J_{0})=\mathcal{H}^{1}(J_{\infty})=0$.

We say that a set $E\in \R^2$ has a decomposition into other sets $\{E_i\}_{i}$ up to $\mathcal{H}^1$-measure zero sets if 
$$\mathcal{H}^1\left(\left(E\setminus \bigcup_{i} E_i\right)\cup \left( \bigcup_{i} E_i\setminus E\right)\right)=0 $$
and $\mathcal H^1(E_i\cap E_j) = 0$ for every $i \ne j$.

For the particular case of planar sets of finite perimeter we have the following decomposition theorem from \cite[Corollary 1]{ACMM2001}.

\begin{theorem} \label{thm:planardecomposition}
Let $E \subset \mathbb R^2$ have finite perimeter. Then, there exists a unique decomposition of $\partial^ME$ into rectifiable Jordan curves $\{C_i^+, C_k^-\,:\,i,k \in \mathbb N\}$, up to $\mathcal H^1$-measure zero sets, such that
\begin{enumerate}
    \item Given $\text{int}(C_i^+)$, $\text{int}(C_k^+)$, $i \ne k$, they are either disjoint or one is contained
in the other; given $\text{int}(C_i^-)$, $\text{int}(C_k^-)$, $i \ne k$, they are either disjoint or one is
contained in the other. Each $\text{int}(C_i^-)$ is contained in one of the $\text{int}(C_k^+)$.
    \item $P(E,\R^2) = \sum_{i}\mathcal H^1(C_i^+) + \sum_k \mathcal H^1(C_k^-)$.
    \item If $\text{int}(C_i^+) \subset \text{int}(C_j^+)$, $i \ne j$, then there is some rectifiable Jordan curve  $C_k^-$ such that $\text{int}(C_i^+)\subset \text{int}(C_k^-) \subset \text{int}(C_j^+)$. Similarly, if $\text{int}(C_i^-) \subset \text{int}(C_j^-)$, $i \ne j$, then there is some rectifiable Jordan curve  $C_k^+$ such that $\text{int}(C_i^-)\subset \text{int}(C_k^+) \subset \text{int}(C_j^-)$.
    \item Setting $L_j =\{i \,:\, \text{int}(C_i^-)\subset \text{int}(C_j^+)\}$ the sets $Y_j = \text{int}(C_j^+) \setminus \bigcup_{i \in L_j}\text{int}(C_i^-)$ are pairwise disjoint, indecomposable and $E = \bigcup_j Y_j$.
\end{enumerate}

\end{theorem}

Since sets of finite perimeter are defined via the total variation of $BV$-functions, they are understood modulo $2$-dimensional measure zero sets. In particular, the last equality in (4) of Theorem \ref{thm:planardecomposition} is modulo measure zero sets.

In order to prove the sufficiency part of Theorem \ref{thm:planar} we will proceed as follows: Starting from a set $E\subset \Omega$ of finite perimeter we first find an extension $E'$ to $\R^2$ using the fact that $\Omega$ is a $BV$-extension domain. Then we  decompose $\partial^M E'$ using Theorem \ref{thm:planardecomposition} and after proving the quasiconvexity of each of the open connected components $\Omega_i$ of $\R^2\setminus \overline{\Omega}$, we will be able to perturb the Jordan curves of the decomposition of $\partial^M E'$ around each $\partial \Omega_i$ so that we get a final set $\widetilde E$ which will be a strong extension of $E$. An application of Theorem \ref{thm:mainRn} will conclude the proof.

We start by presenting a couple of lemmas showing the quasiconvexity of all the connected components of $\R^2\setminus\overline{\Omega}$.

\begin{lemma}\label{lem:quasi.comp.domains} 
Suppose that $\Omega \subset \mathbb R^2$ is a bounded $BV$-extension domain. Then there exists a constant $C>0$ so that for any connected component $\Omega_i$ of $\mathbb R^2 \setminus \overline{\Omega}$, any two points $z,w \in \partial\Omega_i$ can be connected by a curve $\beta \subset \overline{\Omega_i}$ with $\ell(\beta) \le C|z-w|$.
\end{lemma}
 
\begin{proof}
One can essentially follow step by step the proof of \cite[Theorem 1.1]{KMS2010}, once we have taken into account some facts.

\begin{enumerate}
    \item For a given $i$, since $\Omega$ is a $BV$-extension domain, so is $\Omega' = \mathbb R^2\setminus \overline{\Omega}_i$. As an extension operator we can take
\[
T' \colon BV(\Omega') \to BV(\mathbb R^2) \colon 
u \mapsto T(u|_{\Omega})|_{\overline{\Omega}_i} + u,
\]
where $T$ is the extension operator from $BV(\Omega)$ to $BV(\mathbb R^2)$. Let us explain more in detail why  our resulting function $T'u$ is well-defined as a function in $ BV(\R^2)$. Observe that the closures of the different components $\overline{\Omega}_i$ can only intersect between themselves in just one point. That is,
\begin{equation}\label{eq:few.inter.Omega_i}
  \#{\{\partial \Omega_i\cap \partial\Omega_j\}}\leq 1\;\;\;\text{for every}\;\; i\neq j.
\end{equation}
Otherwise we would be losing the connectedness of $\Omega$ or either $\Omega_i$ and $\Omega_j$ are the same component. This means that $\partial\Omega_i\cap \bigcup_{j\neq i} \overline{\Omega}_j $ is a countable set.  Once we are aware of this simple fact it is clear that $T'(u)$ behaves well around $\partial\Omega_i$ and it belongs to $BV(\R^2)$.

Observe that since $\Omega'$ is a $BV$-extension domain there is a constant $C'>0$ for which the property $(PE2)$  of extension of sets of finite perimeter holds. Note that this constant $C'$ only depends on the norm $\|T'\|$, which only depends on $\|T\|$ and which in turn only depends on the constant $C>0$ of the same property $(PE2)$ but now applied to the $BV$-extension domain $\Omega$.

\item We can assume that $\Omega'$ is bounded, and hence also a $BV_l$-extension domain thanks to \cite[Lemma 2.1]{KMS2010}. If $\Omega'$ was not bounded then $\Omega_i$ had to be bounded and we can take a large enough radius $R>0$ so that
$$\Omega_i \subset B(0,R) \;\; \text{and}\;\;\Omega\subset B(0,R)\setminus \overline{\Omega}_i.$$
It is clear that changing $\Omega'$ by $\Omega'\cap B(0,R)$ does not affect the $BV$-extension property.

\end{enumerate}

The proof of \cite[Theorem 1.1]{KMS2010} is made under the assumptions that a set $\Omega'$  is a bounded simply connected $BV_l$-extension domain, reaching as a conclusion that $\R^2\setminus \Omega'$ is quasiconvex.

In the case $\Omega_i$ was unbounded, $\Omega'$ will be a bounded simply connected $BV_l$-extension domain and we apply the previous result directly to show the quasiconvexity of $\overline{\Omega}_i$.

If $\Omega_i$ was bounded, after the modification mentioned above, $\Omega'$ will be a bounded $BV_l$-extension domain. To prove the quasiconvexity of $\overline{\Omega}_i$ in \cite[Theorem 1.1]{KMS2010} the simply connectedness was just used at the following point: when we take two points $z,w\in\partial\Omega_i$ and join them with a line-segment $L_{z,w}$, the set $ \Omega'  \cap L_{z,w}$ consists on the disjoint union of countably many line-segments $L_{z_i,w_i}$, with $z_i,w_i\in\partial\Omega_i$. Now, under the assumption of simply connectedness of $\Omega'$ one can assert that $\Omega' \setminus L_{z_i,w_i}$ has two disjoint connected components. However, in our case this is still true because otherwise  $\Omega_i$ would not be connected.

The previous facts yield that every set $\overline\Omega_i$ is quasiconvex. A careful reading of the proof \cite[Theorem 1.1]{KMS2010} also shows that the constant of quasiconvexity of all these sets is uniformly bounded by a constant $C>0$, independent of $i$. Indeed, the quasiconvexity constant of any set $\overline \Omega_i$  only depends on the constant of the extension property of sets of finite perimeter $(PE2)$ for the $BV$-extension domains $\Omega'=\R^2\setminus \overline \Omega_i$, which, as we already noted, depends only on the constant for the $BV$-extension domain $\Omega$ independently of what $i$ we are fixing.
\end{proof}

Notice that the previous Lemma \ref{lem:quasi.comp.domains} implies, in particular, that if $\Omega$ is a bounded $BV$-extension domain, then all open connected components of $\R^2\setminus \overline{\Omega}$ are Jordan domains.

We record the following general lemma which might be of independent interest. A version of it for quasiconvex sets was proven via conformal maps in \cite{KRZ}. Let us also point out that with the sharp Painleve-length result for a connected set \cite{LPR2020} one could quite easily prove a version of the lemma with a multiplicative constant $2$.
\begin{lemma}\label{lem:quasi.int.Omega_i}
 Let $\Omega$ be a Jordan domain. For every $x,y \in \overline{\Omega}$, every $\varepsilon > 0$ and any rectifiable curve $\gamma \subset \overline{\Omega}$ joining $x$ to $y$ there exists a curve $\sigma \subset \Omega \cup \{x,y\}$ joining $x$ to $y$ so that
my g \[
   \ell(\sigma) \le \ell(\gamma) + \varepsilon.
 \]
\end{lemma}
\begin{proof}
Without loss of generality, we may assume that $\gamma \colon [0,\ell(\gamma)] \to \mathbb R^2$ minimizes the length of curves joining $x$ to $y$ in $\overline{\Omega}$, $\gamma(0) = x$, $\gamma(\ell(\gamma)) = y$, and that $\gamma$ has unit speed.

If $\gamma((0,\ell(\gamma))) \cap \partial\Omega = \emptyset$,
we are done. Suppose this is not the case and define
\[
s_1 = \min \{t \in [0,\ell(\gamma)] \,:\, \gamma(t) \in \partial\Omega\}
\]
and
\[
s_2 = \max \{t \in [0,\ell(\gamma)] \,:\, \gamma(t) \in \partial\Omega\}
\]
If $s_1=s_2$, by minimality the curve $\gamma$ is the concatenation of line-segments $[x,\gamma(s_1)]$ and $[\gamma(s_1), y]$. In this case, for small $r \in (0,\varepsilon /(2\pi))$, the curve $\gamma$ divides $B(\gamma(s_1),r)$ into two parts so that one of them is a subset of $\Omega$. Thus, we may replace part of $\gamma$ by an arc of the circle $S(\gamma(s_1),r)$, and we are done. 

We are then left with the more substantial case where $s_1 < s_2$.
Since $\partial\Omega$ is a Jordan loop, the set $\partial \Omega \setminus \gamma(\{s_1,s_2\})$ consists of two connected components $T_1$ and $T_2$.

We will show that $\gamma$ can be slightly pushed away from $\partial\Omega$ in directions that change in a locally Lipschitz way in $(0,\ell(\gamma))$. Namely, we assert that there exist functions
\begin{align*}
    &\varepsilon\colon (0,\ell(\gamma))\to (0,1),\\
    &v\colon(0,\ell(\gamma))\to \mathbb{S}^1,
\end{align*}
so that $\varepsilon(\cdot)$ and $v(\cdot)$ are locally Lipschitz continuous and satisfy $\gamma(t)+hv(t)\in\Omega$ for all $0<h<\varepsilon(t)$ and $t \in (0,\ell(\gamma))$.

In order to show this, let $t \in (0,\ell(\gamma))$. If $\gamma(t) \in \Omega$, then with $\varepsilon_t = \frac12\dist(\gamma(t),\partial\Omega)$ we have $\gamma(s) + hv \in \Omega$ for all $s\in (t-\varepsilon_t, t+\varepsilon_t)\cap(0,\ell(\gamma))$,  $v \in \mathbb S^1$ and $0 < h <\varepsilon_t$. 

Suppose then that $t \in \gamma^{-1}(\partial\Omega) \cap (0,\ell(\gamma))$.
Without loss of generality we may assume that $\gamma(t) \in T_1$. 
The concatenation of $\gamma|_{[s_1,s_2]}$ with $T_2$ forms a closed  loop $\alpha$ so that one of the components $\tilde\Omega$ of its complement is contained in $\Omega$, and $\gamma(t) \in \partial\tilde\Omega$.
Now, let 
$r_t = \dist(\gamma(t),T_2)$. 
Then, by minimality of $\gamma$, the set $\gamma \cap B(\gamma(t),r_t)$
is contained on the boundary of a convex set $K_t = B(\gamma(t),r_t) \setminus \tilde\Omega$ with non-empty interior.
Consequently, there exists a constant $\varepsilon_t>0$ 
so that for any $t-\varepsilon_t < \tau_1 < \tau_2 < t +\varepsilon_t$ for which the outer normal vectors $w_1$ and $w_2$ to $K_t$ exist at $\gamma(\tau_1)$ and $\gamma(\tau_2)$ respectively, there is a Lipschitz map $[\tau_1,\tau_2] \to \mathbb S^1\colon t \mapsto v_{\tau_1,\tau_2}(t)$ so that
$v_{\tau_1,\tau_2}(\tau_1) = w_1$, $v_{\tau_1,\tau_2}(\tau_2) = w_2$, and 
$\gamma(t)+hv_{\tau_1,\tau_2}(t)\in\Omega$ for all $0<h<\varepsilon_t$ and $\tau_1 \le t \le \tau_2$.

Write $I \subset (0,\ell(\gamma))$ to be the points $t$ where a normal direction to $\gamma$ exists at $\gamma(t)$.
Now, cover $(0,\ell(\gamma))$ with the intervals $(t-\varepsilon_t,t+\varepsilon_t)\cap(0,\ell(\gamma))$ and then take a subcover $\{U_i=(\overline{t}_i-\varepsilon_{\overline{t}_i},\overline{t}_i+\varepsilon_{\overline{t}_i})\}_{i\in\mathbb Z}$ that is finite for compact subsets of $(0,\ell(\gamma))$, and so that every $t\in(0,\ell(\gamma))$ belongs to at most two intervals $U_i$. Assume the intervals $U_i$ are in order, that is $U_i$ only intersects $U_{i-1}$ and $U_{i+1}$. 
By dividing into smaller intervals if needed, we may also assume that if $\gamma(U_i)\cap \partial \Omega \ne \emptyset$ and
$\gamma(U_{i+1})\cap \partial \Omega \ne \emptyset$, then 
$\gamma(U_i\cup U_{i+1})\cap \partial \Omega \subset T_j$ for $j =1$ or $j=2$.
This allows us to select the normal directions $w \in \mathbb S^1$ in a way so that they agree for the intervals $U_i$ and $U_{i+1}$ at the points $t \in I \cap U_i\cap U_{i+1}$.
Notice that for $i$ for which $\gamma(U_i)\cap \partial\Omega = \emptyset$ we have to make a choice between two opposite directions.

We will then have  subset $I \subset (0,\ell(\gamma))$ with $\mathcal H^1((0,\ell(\gamma)) \setminus I) = 0$, and an open covering $\{U_i\}_i$ of $(0,\ell(\gamma))$ of multiplicity at most two, where $U_i$ are intervals, so that
\begin{itemize}
    \item for every $i \in \mathbb Z$ there exists a constant $\varepsilon_i>0$ so that for every $\tau_1,\tau_2 \in I \cap  U_i$, $\tau_1 < \tau_2$, there is a Lipschitz map $[\tau_1,\tau_2] \to \mathbb S^1\colon t \mapsto v_{\tau_1,\tau_2}(t)$ so that 
$\gamma(t)+hv_{\tau_1,\tau_2}(t)\in\Omega$ for all $0<h<\varepsilon_i$ and $\tau_1 \le t \le \tau_2$,
\item if $v_{\tau_1,\tau_2}$ and $v_{\tau_2,\tau_3}$ have been defined as above, $v_{\tau_1,\tau_2}(\tau_2) = v_{\tau_2,\tau_3}(\tau_2)$. 
\end{itemize}

    For each $i\in \mathbb Z$ we will now fix a $t_i \in I \cap  U_i \cap U_{i+1}$, and define  $v(t) = v_{t_i,t_{i+1}}(t)$ on $[t_i,t_{i+1}]$. 
A locally Lipschitz choice for $\varepsilon$ can be given by defining
\[
\varepsilon(t) = \frac{t_{i+1}-t}{t_{i+1}-t_i}\min(\varepsilon_i,\varepsilon_{i+1})  + \frac{t-t_{i}}{t_{i+1}-t_i}\min(\varepsilon_{i+1},\varepsilon_{i+2})
\]
when $t \in [t_i,t_{i+1}]$.

Let $i_0 \in \N$ be such that $i_0 > \frac{2}{s_2-s_1} \geq \frac{2}{\ell(\gamma)}$. Then, for any $i \ge i_0$, the function
$v(\cdot)$ is Lipschitz in $[1/i, \ell(\gamma)-1/i]$ and $\eta_i := \min_{t \in [1/i, \ell(\gamma)-1/i]}\varepsilon(t) > 0$. Hence, if we define
\[
\delta_i(t) = \begin{cases}
\eta_i\frac{\min\{|\ell(\gamma) - 1/i-t|,|1/i-t|\}}{\ell(\gamma) - 2/i}, & \text{if } t\in [1/i, \ell(\gamma)-1/i],\\
0, & \text{otherwise}
\end{cases},
\]
we have $|\delta'_i(t)|=\frac{\eta_i}{\ell(\gamma)-2/i}$ for all $t\in (1/i, \ell(\gamma)-1/i)\setminus \{\ell(\gamma)/2\}$, and also  $|\delta_i(t)|\leq \eta_i/2\leq \varepsilon(t)/2$ for every $t$. Now if we let
\[
L_i = \int_{0}^{\ell(\gamma)}|(\delta_i(t)v(t))'|\,dt < \infty,
\] 
defining
\[
\delta(t) = \sum_{i= i_0}^\infty\frac{2^{-i-1}\min(\varepsilon,1)}{1+L_i}\delta_i(t),
\]
we get a function $\delta \colon [0,\ell(\gamma)] \to \mathbb R$ such that $\delta(0)=\delta(\ell(\gamma))=0$ and $0 < \delta(t) < \varepsilon(t)$ for all $t \in (0,\ell(\gamma))$. Note that the function $\delta$ is continuous as a limit of  an absolutely and uniformly convergent series of continuous
functions, and it is differentiable except on $\ell(\gamma)/2$.
Thus, $\sigma \colon [0,\ell(\gamma)] \to \mathbb R^2$ defined by
\[
\sigma(t) = \gamma(t) + \delta(t)v(t)
\]
is a curve joining $x$ and $y$, $\sigma((0,\ell(\gamma))) \subset \Omega$ and
\begin{align*}
\ell(\sigma) &= \int_{0}^{\ell(\gamma)}|\sigma'(t)|\,dt \le \int_{0}^{\ell(\gamma)}|\gamma'(t)|\,dt + \int_{0}^{\ell(\gamma)}|(\delta(t)v(t))'|\,dt
\\
& \le \ell(\gamma) + \sum_{i=i_0}^\infty\frac{2^{-i-1}\varepsilon}{1+L_i}\int_{0}^{\ell(\gamma)}|(\delta_i(t)v(t)7)'|\,dt < \ell(\gamma) + \varepsilon,
\end{align*}
finishing the proof.
\end{proof}

The next lemma, together with Theorem \ref{thm:planardecomposition}, are the key tools for our proof of the sufficiency part of Theorem \ref{thm:planar}, that we will show afterwards.

\begin{lemma}\label{lma:planarpushing} 
 Let $\Omega \subset \mathbb R^2$ be a bounded $BV$-extension domain  and $\Omega_i$ the open connected components of $\R^2\setminus \overline{\Omega}$. Suppose that the set $H=\partial \Omega\setminus \bigcup_i \overline{\Omega}_i $ is purely $1$-unrectifiable and
 let $E \subset \mathbb R^2$ be a Jordan domain with $\partial E$ rectifiable. Then there exists a  set $\widetilde E \subset \mathbb R^2$ of finite perimeter so that
 \begin{itemize}
     \item[(i)] $E \cap \Omega = \widetilde E \cap \Omega$,
     \item[(ii)] $\mathcal H^1(\partial^M \widetilde E) \le C \mathcal H^1(\partial^M E)$, and
     \item[(iii)] $\mathcal H^{1}(\partial^M \widetilde E \cap \partial \Omega) = 0$,
 \end{itemize}
 where the constant $C$ is absolute.
\end{lemma}
\begin{proof}
Consider the at most countably many components $\{\Omega_i\}_i$ of $\R^2 \setminus \overline{\Omega}$. For each $i$ we want to modify the set $E$ in $\overline{\Omega_i}$ to get some $\widetilde E\subset\R^2$  with $\partial \widetilde E$ rectifiable so that 
 \begin{equation}\label{eq:boundarytozero}
  \mathcal{H}^1(\partial^M \widetilde E \cap \partial \Omega_i) = 0
 \end{equation}
 and
 \begin{equation}\label{eq:boundary.int.controled}
 \mathcal{H}^1(\partial^M \widetilde E \cap \Omega_i) \le C \mathcal{H}^1(\partial^M E \cap \overline{\Omega_i}) .
 \end{equation}
 Let us show how to conclude the proof of the lemma after assuming these facts. Since we are not changing the set $E$ inside $\Omega$ the property (i) is clear. To check (ii) let us first write
 $$\mathcal{H}^1(\partial ^M \widetilde E)= \mathcal{H}^1(\partial ^M \widetilde E\cap\Omega)+\mathcal{H}^1(\partial ^M \widetilde E\cap\partial \Omega)  + \mathcal{H}^1(\partial ^M \widetilde E\cap (\R^n\setminus \overline{\Omega})).$$
 We will estimate each of these terms separately. For the first one is clear that $\mathcal{H}^1(\partial ^M \widetilde E\cap\Omega)=\mathcal{H}^1(\partial ^M E\cap\Omega)$. For the second one we use the fact that $\partial \widetilde E$ is rectifiable, that $\partial \Omega\setminus \bigcup_i \overline{\Omega}_i$ is purely $1$-unrectifiable and \eqref{eq:boundarytozero},
 \begin{align}\label{eq:mes.zero.in.boundary}
     \mathcal{H}^1(\partial ^M \widetilde E\cap\partial \Omega)&=
     \mathcal{H}^1\left(\partial ^M \widetilde E\cap\left[\partial \Omega  \setminus \bigcup_i \overline{\Omega}_i\right]\right)+
     \mathcal{H}^1\left(\partial ^M \widetilde E\cap\left[\partial \Omega  \cap\bigcup_i \overline{\Omega}_i\right]\right) \nonumber \\ 
     &=\mathcal{H}^1\left(\bigcup_i(\partial ^M \widetilde E\cap \partial \Omega  \cap\overline{\Omega}_i)\right) \nonumber \\
     &\leq \sum_i \mathcal{H}^1(\partial^M \widetilde E \cap \partial \Omega_i)=0.
 \end{align}
 For the third term we use \eqref{eq:boundary.int.controled} to get
 $$\mathcal{H}^1(\partial ^M \widetilde E\cap (\R^n\setminus \overline{\Omega})) \leq \sum_i\mathcal{H}^1(\partial^M \widetilde E\cap \Omega_i)\leq C \sum_i \mathcal{H}^1(\partial^M E\cap \overline{\Omega_i}) .$$
All these estimates together yield
$$ \mathcal{H}^1(\partial ^M \widetilde E)\leq \mathcal{H}^1(\partial ^M E\cap\Omega)  +C \sum_i \mathcal{H}^1(\partial^M E\cap \overline{\Omega_i}) .$$
Since $\{x\in\partial \Omega_i:\, x\in \partial \Omega_j \;\;\text{for some}\;\; j\neq i\}$ is at most countable by \eqref{eq:few.inter.Omega_i}, we conclude that
$$ \mathcal{H}^1(\partial ^M \widetilde E)\leq C\mathcal{H}^1(\partial ^M E),$$
proving (ii). Finally (iii) has already been shown in \eqref{eq:mes.zero.in.boundary}.

We now move to prove how to modify $E$ inside each set $\overline{\Omega}_i$ in order to get \eqref{eq:boundarytozero} and \eqref{eq:boundary.int.controled}.

If $\mathcal H^1(\partial^ME\cap \partial \Omega_i) = 0$, we may skip this $i$ and move to the next. Let us thus assume $\mathcal H^1(\partial^ME\cap \partial \Omega_i) > 0$.
Let $f \colon \mathbb S^1 \to \partial E$ be a parameterisation 
of the boundary by a homeomorphism.
By the Lebesgue density theorem, for almost every $t \in f^{-1}(\partial^ME\cap \partial \Omega_i)$ there exists a $r_t>0$
so that for all $0 < r < r_t$
\begin{equation}\label{eq:balllarge}
 \mathcal H^1\left(f(B(t,r))\cap \partial\Omega_i\right) \ge \frac12 \mathcal H^1(f(B(t,r))).
 \end{equation}
 By the Vitali covering lemma, we then find a disjointed collection $\{B(t_j,r_j)\}_j$ so that \eqref{eq:balllarge} holds for each of the balls and
 \[
 \mathcal H^1\left((\partial^ME\cap \partial \Omega_i) \setminus \bigcup_j f(B(t_j,r_j)\right) = 0.
 \]
 Now, we define $I_{i,j} = \overline{B(t_j,r_j)} \cap\mathbb S^1$ for each $j$ and obtain a collection $\{I_{i,j}\}_j$ of closed arcs in $\mathbb S^1$
 whose interiors are pairwise disjoint,

 \[
 \mathcal H^1\left((\partial^ME\cap \partial \Omega_i) \setminus \bigcup_j f(I_{i,j})\right) = 0
 \]
 and
 \[
 \mathcal H^1\left(f(I_{i,j})\cap \partial\Omega_i\right) \ge \frac12 \mathcal H^1(f(I_{i,j}))
 \]
 for every $j$.
 \begin{figure}
     \centering
     \includegraphics[width=0.75\textwidth]{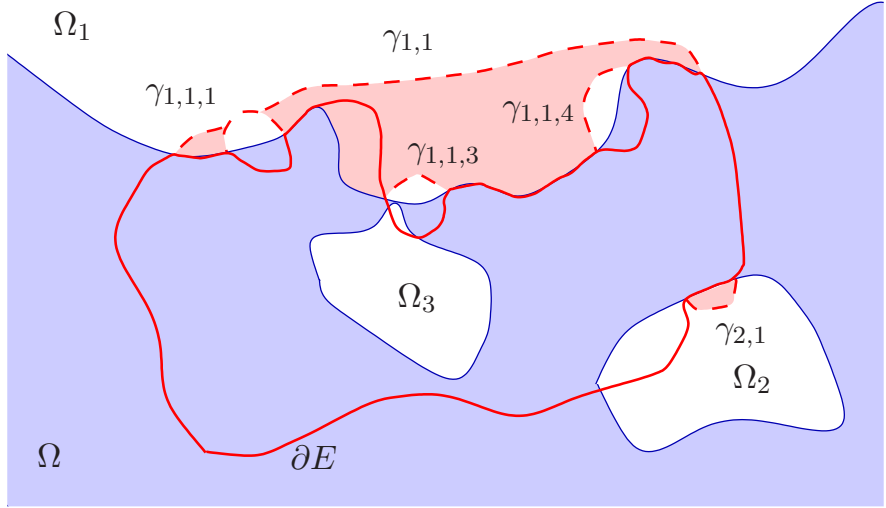}
     \caption{An illustration of the construction in Lemma \ref{lma:planarpushing}. The boundary $\partial E$ intersect the boundaries $\partial \Omega_1$ and $\partial\Omega_2$ in a set of positive $\mathcal H^1$-measure. The modification of $E$ inside $\Omega_1$ consists of the added set bounded by $\gamma_{1,1}$ from which three sets have been removed, bounded by $\gamma_{1,1,1}$, $\gamma_{1,1,3}$, and $\gamma_{1,1,4}$, respectively. The modification inside $\Omega_2$ consists of only one added part bounded by $\gamma_{2,1}$.}
     \label{fig:perturbation}
 \end{figure}
 
For the next argument we have $i,j$ fixed. The set $f(I_{i,j}) \setminus \partial \Omega_i$ consists of at most countably many open curves $\{\alpha_{i,j,k}\}_{k}$. For each $k $ for which $\alpha_{i,j,k} \cap \Omega_i = \emptyset$, we use Lemma \ref{lem:quasi.comp.domains} to find a curve $\beta_{i,j,k} \subset \overline{\Omega_i}$ such that $\ell(\beta_{i,j,k}) \le C|z_{i,j,k}-w_{i,j,k}|$ where $z_{i,j,k}$ and $w_{i,j,k}$ are the endpoints of $\alpha_{i,j,k}$. Now, for $\varepsilon=|z_{i,j,k}-w_{i,j,k}| $, Lemma \ref{lem:quasi.int.Omega_i} provides us with another curve $\gamma_{i,j,k}\subset \Omega_i\cup \{z_{i,j,k},w_{i,j,k}\}$ so that
\begin{equation}
    \ell(\gamma_{i,j,k})\leq \ell(\beta_{i,j,k})+|z_{i,j,k}-w_{i,j,k}|\leq (C+1)|z_{i,j,k}-w_{i,j,k}|.
\end{equation}
The curves $\alpha_{i,j,k}$ and $\gamma_{i,j,k}$ enclose a bounded subset that we call $E_{i,j,k} \subset \mathbb R^2$. 
Similarly, if we let $z_{i,j}$ be the first, and $w_{i,j}$ the last point of $f(I_{i,j}) \cap \partial \Omega_i$  we again use Lemmas \ref{lem:quasi.comp.domains} and \ref{lem:quasi.int.Omega_i}  to connect $z_{i,j}$ to $w_{i,j}$ with a curve $\gamma_{i,j} \subset \Omega_i \cup\{z_{i,j},w_{i,j}\}$ so that 
\begin{equation}
\ell(\gamma_{i,j}) \le (C+1)|z_{i,j}-w_{i,j}|.
\end{equation}
Let $F_{i,j}$ be the bounded set enclosed by $\partial \Omega_i$ (from $z_{i,j}$ to $w_{i,j}$) and by $\gamma_{i,j}$. Now, we will modify $E$ by considering
 $$\widetilde E_{i,j}=E\cup \left( F_{i,j} \setminus \bigcup_{k} E_{i,j,k} \right) .$$
 See Figure \ref{fig:perturbation} for an illustration of the modification.
 
 Repeating this process for all $i$ with $\mathcal H^1(\partial^ME\cap \partial \Omega_i) > 0$ and all $j$ we can finally define
 $$\widetilde E=\bigcup_{i,j}\widetilde E_{i,j} .$$
 Let us check that the properties $\eqref{eq:boundarytozero}$ and $\eqref{eq:boundary.int.controled}$ hold.
 Firstly, observing that we did not modified $\partial E$ outside the arcs $f(I_{i,j})$,
 \begin{align*}
     \mathcal{H}^1(\partial^M \widetilde E\cap \partial\Omega_i)&=\mathcal{H}^1\left((\partial^M \widetilde E\cap \partial \Omega_i)\setminus\bigcup_j f(I_{i,j})\right)\\
     &\qquad +\sum_j\mathcal{H}^1(\partial^M \widetilde E \cap \partial\Omega_i\cap f(I_{i,j})  )\\
     &=\mathcal{H}^1\left((\partial^M  E\cap \partial \Omega_i)\setminus\bigcup_j f(I_{i,j})\right)\\
     &\qquad+\sum_j \left(\mathcal{H}^{1}(\partial^M F_{i,j} \cap \partial\Omega_i) +\sum_k  \mathcal{H}^{1}(\partial^M E_{i,j,k} \cap \partial\Omega_i)\right) \\
     &=0,
 \end{align*}
 which gives us \eqref{eq:boundarytozero}.
 Secondly,
 \begin{align*}
     \mathcal{H}^1(\partial^M \widetilde E \cap \Omega_i) &\leq \mathcal{H}^1(\partial^M E\cap\Omega_i)+\sum_j\left(    
     \mathcal{H}^1(\gamma_{i,j})+\sum_k \mathcal{H}^1(\gamma_{i,j,k})\right) \\
     &\leq \mathcal{H}^1(\partial^M E\cap\Omega_i)+\sum_j\left(    
     (C+1)|z_{i,j}-w_{i,j}|+\sum_k (C+1)|z_{i,j,k}-w_{i,j,k}|\right) \\
     &\leq  \mathcal{H}^1(\partial^M E\cap\Omega_i)+\sum_j 2C \mathcal{H}^1(f(I_{i,j})) \\
     &\leq \mathcal{H}^1(\partial^M E\cap\Omega_i)+\sum_j 2(C+1) \mathcal{H}^1(f(I_{i,j})\cap \partial \Omega_i) \\
     &\le C \mathcal{H}^1(\partial^M E \cap \overline{\Omega}_i)
 \end{align*}
 proving \eqref{eq:boundary.int.controled}. 
 \end{proof}

\begin{proof}[Proof of Theorem \ref{thm:planar}]
One direction is proven in Corollary \ref{cor:unrectifiabilitynecessity}. Thus we only need to prove the converse. Thus, assume that 
$\Omega \subset \mathbb R^2$ is a bounded $BV$-extension domain and  that the set $H=\partial \Omega\setminus \bigcup_i \overline{\Omega}_i $ is purely $1$-unrectifiable, where $\Omega_i$ are the open connected components of $\R^2\setminus \overline{\Omega}$.

 We will show that $\Omega$ has the strong extension property for sets of finite perimeter and hence, by Theorem \ref{thm:mainRn}, $\Omega$ will be a $W^{1,1}$-extension domain. Using the fact that $\Omega$ is a bounded $BV$-extension domain if we let $E \subset \Omega$ be a set of finite perimeter in $\Omega$ then there exists an extension  $E'$  to $\mathbb R^2$ so that $P(E',\mathbb R^2) \le C P(E,\Omega)$. This extension can be obtained for instance by the Maz'ya and Burago result \cite[Section 9.3]{mazya}.  
 
 Let now $\{C_i^+, C_k^-\,:\,i,k \in \mathbb N\}$ be the rectifiable Jordan curves of Theorem \ref{thm:planardecomposition} for the set $E'$. By applying Lemma \ref{lma:planarpushing},  each Jordan domain $\text{int}(C_i^+)$ can be replaced by a  set $\widetilde E_i^+$ so that
 $\widetilde E_i^+ \cap \Omega = \text{int}(C_i^+) \cap \Omega$, $\mathcal H^1(\partial^M \widetilde E_i^+) \le C \mathcal H^1(C_i^+)$,
 and $\mathcal H^1(\partial^M\widetilde E_i^+ \cap \partial \Omega) = 0$.
 Similarly, each   $\text{int}(C_k^-)$ can be replaced by a set $\widetilde E_k^-$ so that
 $\widetilde E_k^- \cap \Omega = \text{int}(C_k^-) \cap \Omega$, $\mathcal H^1(\partial ^M\widetilde E_k^-) \le C \mathcal H^1(C_k^-)$, and
 $\mathcal H^1(\partial^M\widetilde E_k^- \cap \partial \Omega) = 0$.

 Now, 
 \[
 E = E' \cap \Omega = \left(\bigcup_{i}\text{int}(C_i^+) \setminus \bigcup_k \text{int}(C_k^-) \right) \cap \Omega = \left(\bigcup_{i}\widetilde E_i^+ \setminus \bigcup_k \widetilde E_k^- \right) \cap \Omega,
 \]
holds modulo a measure zero set. Thus, the set
 \[
 \widetilde E = \left(\bigcup_{i}\widetilde E_i^+ \setminus \bigcup_k \widetilde E_k^- \right)
 \]
 is an extension of $E$ to $\mathbb R^2$,
 and
 \begin{align*}
 P(\widetilde E,\mathbb R^2) & \le \sum_{i}\mathcal H^1(\partial^M\widetilde E_i^+) + \sum_k \mathcal H^1(\partial^M\widetilde E_k^-)\\
 & \le \sum_{i}C\mathcal H^1(C_i^+) + \sum_k C\mathcal H^1(C_k^-)\\
 & = CP(E',\mathbb R^2) \le C P(E,\Omega).
 \end{align*}
 Since,
 \[
 \mathcal H^1(\partial^M \widetilde E \cap \partial \Omega) \le \sum_i
 \mathcal H^1(\partial^M \widetilde E_i^+ \cap \partial \Omega) + 
 \sum_k
 \mathcal H^1(\partial^M \widetilde E_k^- \cap \partial \Omega) = 0,
 \]
 the set $\widetilde E$ is the strong extension of $E$ that we had to find.
\end{proof}

\section*{Acknowledgements}
The authors thank Panu Lahti for several comments on an earlier version of this paper.

\end{document}